\documentclass[11pt, letterpaper]{amsart}
\usepackage[plainpages=false,pdfpagelabels,
colorlinks=true,linkcolor=blue,
citecolor=blue,filecolor=blue,      
urlcolor=blue,hypertexnames=false]{hyperref}
\usepackage{amssymb}
\usepackage{graphics}
\usepackage{graphicx}               
\usepackage{latexsym}
\usepackage{amsmath}
\usepackage{amssymb,amsthm,amsfonts}
\usepackage{amscd}
\usepackage[arrow, matrix, curve]{xy}
\usepackage{syntonly}
\usepackage{tikz}
\usetikzlibrary{quotes}
\usepackage{tikz-cd}
\usepackage{enumitem}
\usepackage{todonotes}
\usepackage{hyperref,cleveref}

\usepackage[small,nohug,heads=vee]{diagrams}
\diagramstyle[labelstyle=\scriptstyle]
\usepackage{cancel}
\makeatletter
\@namedef{subjclassname@2020}{\textup{2020} Mathematics Subject Classification}
\makeatother

\newcommand{\cCV}{\mathcal{CV}}
\newcommand{\Spec}{\mathrm{Spec}}

\usepackage{ulem}

\theoremstyle{plain}
\newtheorem{thm}{Theorem}[section]
\newtheorem{theorem}[thm]{Theorem}
\newtheorem*{theorem*}{Main Theorem}
\newtheorem{lemma}[thm]{Lemma}
\newtheorem{corollary}[thm]{Corollary}

\theoremstyle{definition}

\newtheorem{remark}[thm]{Remark}

\newtheorem{conjecture}[thm]{Conjecture}

\title[On the normality of commuting scheme]{On the normality of commuting scheme for general linear Lie algebra}

\author{Artan Sheshmani \textsuperscript{1,2,3}}
\email{artan@mit.edu}
\author{Xiaopeng Xia \textsuperscript{1}}
\email{xiaxiaopeng@bimsa.cn}
\author{Beihui Yuan \textsuperscript{1}}
\email{beihuiyuan@bimsa.cn}
\address{\textsuperscript{1}Beijing Institute of Mathematical Sciences and Applications, No. 544, Hefangkou Village, Huaibei Town, Huairou District, Beijing 101408}
\address{\textsuperscript{2}Massachusetts Institute of Technology, IAiFi Institute, 77 Massachusetts Ave, 26-555. Cambridge, MA 02139}
\address{\textsuperscript{3}National Research University, Higher School of Economics, Russian Federation, Laboratory of Mirror Symmetry, NRU HSE, 6 Usacheva str., Moscow, Russia, 119048}

\date{}

\begin{document}
\begin{abstract}
The commuting scheme $\mathfrak{C}^d_{\mathfrak{g}}$ for reductive Lie algebra $\mathfrak{g}$ over an algebraically closed field $\mathbb{K}$ is the subscheme of $\mathfrak{g}^d$ defined by quadratic equations, whose $\mathbb{K}$-valued points are $d$-tuples of commuting elements in $\mathfrak{g}$ over $\mathbb{K}$. There is a long-standing conjecture that the commuting scheme $\mathfrak{C}^d_{\mathfrak{g}}$ is reduced. Moreover, a higher-dimensional analogue of Chevalley restriction conjecture was conjectured by Chen-Ng\^{o}. We show that the commuting scheme $\mathfrak{C}^2_{\mathfrak{gl}_n}$ is Cohen–Macaulay and normal.
As corollary, we prove a 2-dimensional Chevalley restriction theorem for general linear group in positive characteristic.
 \end{abstract}
\maketitle 

\tableofcontents

\section{Introduction}

\subsection{Commuting scheme}\ 

Let $G$ be a reductive group over an algebraically closed field $\mathbb{K}$ with Lie algebra $\mathfrak{g}$. For an integer $d\geq 2$, let $\mathfrak{C}^d_{\mathfrak{g}}\subset \mathfrak{g}^d$ be the commuting scheme, which is defined as the scheme-theoretic fiber of the commutator map over the zero
\[
\mathfrak{g}^d\rightarrow \prod\limits_{i<j} \mathfrak{g}, \ \ (x_1,\cdots, x_d)\mapsto \prod_{i<j}[x_i, x_j].
\] 
Its underlying variety (the reduced induced closed subscheme) $\mathfrak{C}^d_{\mathfrak{g},red}$ is called the commuting variety. As a set, $\mathfrak{C}^d_{\mathfrak{g},red}$ consists of $d$-tuples $(x_1,\cdots, x_d)\in \mathfrak{g}^d$ such that $[x_i,x_j]=0$, for all  $1\leq i,j\leq d$.

There is a long-standing conjecture
\begin{conjecture}\label{conj: reduced}
The commuting scheme $\mathfrak{C}^d_{\mathfrak{g}}$ is reduced. 
\end{conjecture}

For $d=2$ and $\mathfrak{gl}_n$, Motzkin-Taussky \cite{Motzkin-Taussky} and Gerstenhaber \cite{Gerstenhaber} independently proved that the variety $\mathfrak{C}^2_{\mathfrak{gl}_n,red}$ is irreducible.
For $d=2$, Richardson in 1979 \cite{Richardson} proved that the variety $\mathfrak{C}^2_{\mathfrak{g},red}$ is irreducible in characteristic $0$. In positive characteristics, Levy \cite{Levy} showed that $\mathfrak{C}^2_{\mathfrak{g},red}$ is irreducible under certain mild assumptions on $G$. However, for $d\geq 3$, $\mathfrak{C}^d_{\mathfrak{g}}$ is in general reducible.

Moreover, Popov \cite{Popov} showed that the singular locus of $\mathfrak{C}^2_{\mathfrak{g},red}$ has codimension $\geq 2$ in characteristic $0$, and Ngo \cite{Ngo-commuting} proved that the variety $\mathfrak{C}^d_{\mathfrak{gl}_2,red}$ is Cohen-Macaulay and normal in characteristics $\neq 2$.
However, Conjecture \ref{conj: reduced} is in general not verified.

On the other hand, we can consider the categorical quotient $\mathfrak{C}^d_{\mathfrak{g}}/\!\!/G$.
Here $G$ acts on $\mathfrak{g}^d$ by the diagonal adjoint action, and the subscheme $\mathfrak{C}^d_{\mathfrak{g}}$ of $\mathfrak{g}^d$ is stable under $G$, which means that the restriction of the action $\mu:G\times_\mathbb{K}\mathfrak{g}^d\to \mathfrak{g}^d$ to $G\times_\mathbb{K}\mathfrak{C}^d_{\mathfrak{g}}$ factors through $\mathfrak{C}^d_{\mathfrak{g}}\hookrightarrow \mathfrak{g}^d$. 
Let $T$ be a maximal torus of $G$ and $\mathfrak{t}$ be the Lie algebra of $T$, and let the Weyl group $W:=N_G(T)/T$ act on $\mathfrak{t}^d$ diagonally. The embedding $\mathfrak{t}^d\hookrightarrow \mathfrak{g}^d$ factors through $\mathfrak{C}^d_{\mathfrak{g}}$ and induces the natural morphism 
\[\Phi: \mathfrak{t}^d/\!\!/W \rightarrow \mathfrak{C}^d_{\mathfrak{g}}/\!\!/ G.\]

In studying the Hitchin morphism from the moduli stack of principal $G$-Higgs bundles on a proper smooth variety $X$ of dimension $d\ge 2$, Chen and Ng\^{o} \cite{Chen-Ngo} are led to

\begin{conjecture}[Chen-Ng\^{o}]\label{conj}
The morphism $\Phi: \mathfrak{t}^d/\!\!/W\rightarrow \mathfrak{C}^d_{\mathfrak{g}}/\!\!/G$
is an isomorphism. 
\end{conjecture}
 When $d=1$ and $\rm{char} \ \mathbb{K} =0$, Conjecture \ref{conj} is simply the classical Chevalley restriction theorem. Since in the context of Higgs bundles, $d$ is the dimension of the underlying variety $X$, we view Conjecture \ref{conj} as a higher-dimensional analogue of Chevalley restriction theorem.

If $\rm{char} \ \mathbb{K} =0$, Conjecture \ref{conj} is  known  to hold for  $G=GL_n(\mathbb{K})$ (Vaccarino \cite{Vaccarino}, Domokos \cite{Domokos}, and later Chen-Ng\^{o} \cite{Chen-Ngo} independently; see also Gan-Ginzburg \cite{Gan-Ginzburg} for case $d=2$), for $G=Sp_{n}(\mathbb{K})$ (Chen-Ng\^{o} \cite{Chen-Ngo-Symplectic}), for $G=O_n(\mathbb{K}),SO_n(\mathbb{K})$ (Song-Xia-Xu \cite{S-X-X}) and for connected reductive groups (Li-Nadler-Yun \cite{L-N-Y:commuting stack} for case $d=2$). A weaker version $\mathfrak{t}^d/\!\!/W\xrightarrow{\sim} \mathfrak{C}^d_{\mathfrak{g},red}/\!\!/G$ is proved by Hunziker \cite{Hunziker} if $G$ is of type $A,B,C,D$ or $G_2$.
If $\rm{char}\ \mathbb{K} >0$, Conjecture \ref{conj} is largely open. However, the weaker version $\mathfrak{t}^d/\!\!/W\xrightarrow{\sim} \mathfrak{C}^d_{\mathfrak{g},red}/\!\!/G$ is proved by Vaccarino \cite{Vaccarino} for $G=GL_n(\mathbb{K})$ and by Song-Xia-Xu \cite{S-X-X} for $G=Sp_n(\mathbb{K}),O_n(\mathbb{K}),SO_n(\mathbb{K})$.

\subsection{Main result and Strategy}\label{sec:main_result_and_strategy}\

Our main result is the following:
\begin{theorem*}\label{main thm in introduction}
Suppose $n\geq 2$, $d=2$, and $G=GL_n(\mathbb{K})$. Then $\mathfrak{C}^2_{\mathfrak{g}}$ is Cohen–Macaulay, integral (i.e. reduced and irreducible) and normal.
\end{theorem*}
\begin{remark}
    Suppose that $\mathrm{char}\ \mathbb{K}$ is good for $G$, $G$ is connected and the derived subgroup of $G$ is simply connected.
    Premet \cite{Premet} showed that the nilpotent variety $\mathfrak{C}^{\mathrm{nil}}(\mathfrak{g})=\{(x,y)\in \mathcal{N}(\mathfrak{g})\times \mathcal{N}(\mathfrak{g})\mid [x,y]=0\}$ is equidimensional, where $\mathcal{N}(\mathfrak{g})$ is the nilpotent variety of $\mathfrak{g}$. Our main theorem  is different from this result since we do not use the nilpotent assumption in our construction.  Furthermore, our result  is regarding the proof of Cohen-Macaulay, reduced and normal properties.
\end{remark}

The main theorem implies that $\mathfrak{C}^2_{\mathfrak{g}}/\!\!/G$ is reduced for $G=GL_n(\mathbb{K})$. By Vaccarino \cite[Theorem 3]{Vaccarino},  we see that
\begin{corollary}\label{cor: restriction}
Suppose $n\geq 2$, $d=2$, $\rm{char}\ \mathbb{K} >0$ and $G=GL_n(\mathbb{K})$. Then $\Phi:\mathfrak{t}^2/\!\!/W\rightarrow \mathfrak{C}^2_{\mathfrak{g}}/\!\!/G$ is an isomorphism.
\end{corollary}

\begin{remark}
Conjecture \ref{conj: reduced} and Conjecture \ref{conj} are very important. Conjecture \ref{conj} imply that there exists a $G$-invariant morphism $\mathrm{sd}: \mathfrak{C}^d_{\mathfrak{g}}\to \mathfrak{t}^d/\!\!/W$, which is called the universal spectral data morphism in \cite{Chen-Ngo} and make the following diagram commute:
\[\begin{tikzcd}
	\mathfrak{t}^d & \mathfrak{C}^d_{\mathfrak{g}} \\
	\mathfrak{t}^d/\!\!/W & \mathfrak{C}^d_{\mathfrak{g}}/\!\!/G
	\arrow[from=1-1, to=1-2]
	\arrow[from=1-1, to=2-1]
	\arrow["\mathrm{sd}"', from=1-2, to=2-1]
	\arrow[from=1-2, to=2-2]
	\arrow[from=2-1, to=2-2]
\end{tikzcd}\]
The existence of $sd$ would be important to the study of the Hitchin morphism.
Our main result implies the Cohen-Macaulay and reduced property of $\mathfrak{C}^2_{\mathfrak{gl}_n}$, thus implies Corollary \ref{cor: restriction} by \cite[Theorem 3]{Vaccarino}.
\end{remark}

For any $n\geq 1$, let
 \begin{equation*}
     R(n)=\mathbb{K}[x_{ij},y_{ij}\mid 1\leq i,j\leq n]/I(n),
 \end{equation*} 
 where $I(n)$ is the ideal generated by all of the entries of the matrix $[X_n,Y_n]$, where $X_n=(x_{ij})_{1\leq i,j\leq n},Y_n=(y_{ij})_{1\leq i,j\leq n}$. 

Then $\mathfrak{C}^2_{\mathfrak{gl}_n}=\mathrm{Spec}R(n)$.  Hence, the main theorem is an immediate consequence of the following:
 \begin{theorem}\label{R(n)}
     The ring $R(n)$ is Cohen–Macaulay and normal for $n\geq 1$.
 \end{theorem}

We prove Theorem \ref{R(n)} by induction on $n$. The proof is clear for $n=1$. We assume that $R(n-1)$ is Cohen–Macaulay and normal and prove that $R(n)$ is Cohen–Macaulay and normal. To complete this induction step, we construct several intermediate rings. They are
\begin{itemize}
    \item $\Tilde{R}(n)=\mathbb{K}[x_{ij},y_{ij},u_i,v_i,t_1,t_2,t\mid 1\leq i,j\leq n]/J(n)$, where $J(n)$ is the ideal of $\Tilde{R}(n)$ generated by all of the entries of the matrices 
    \[
    [X_n,Y_n]-t\begin{pmatrix}
    u_1\\
    \vdots\\
    u_n
    \end{pmatrix}\left(v_1,\dots,v_n\right),
    (X_n-t_{1}I_n)\begin{pmatrix}
    u_1\\
    \vdots\\
    u_n
    \end{pmatrix},\left(v_1,\dots,v_n\right)(Y_n-t_{2}I_n),
    \]
    where $I_n$ is the identity matrix.
    \item $R_1(n)=R(n)/(x_{in}\mid 1\leq i\leq n-1)$.
    \item $R'(n)=\mathbb{K}[x_{ij},y_{ij},u_i,v_i,v_i',t_1,t_2,t_3\mid 1\leq i,j\leq n]/J'(n)$, where $J'(n)$ is the ideal generated by all of the entries of the matrices 
    \[
    [X_n,Y_n]-\begin{pmatrix}
    u_1\\
    \vdots\\
    u_n
    \end{pmatrix}\left(v_1,\dots,v_n\right),
    (X_n-t_{1}I_n)\begin{pmatrix}
    u_1\\
    \vdots\\
    u_n
    \end{pmatrix},\left(v_1,\dots,v_n\right)(X_n-t_{2}I_n),
    \]
    $\left(v_1,\dots,v_n\right)(Y_n-t_{3}I_n)$ and $\left(v_1',\dots,v_n'\right)(Y_n-t_{3}I_n)$.
    \item $R_2(n)=R'(n)/(\det(X_n-t_{2}I_n))$.
\end{itemize}
These rings have the relation
\begin{equation*}
R'(n)/(v_{1},\dots,v_{n})\simeq\Tilde{R}(n)[t'_{2}]/(t).
\end{equation*}
We prove that some of their local rings inherit Cohen-Macaulayness and reducedness from $R(n-1)$.

The intermediate ring $\Tilde{R}(n-1)$ is related to our target ring $R(n)$ by the formula
\[
\Tilde{R}(n-1)/(t-1)\cong R(n)/(x_{in},y_{ni}\mid 1\leq i\leq n-1),
\]
which we will prove in \Cref{lemma:connection_among_rings}. 
Hence,
\[\begin{tikzcd}
	{R(n-1) \text{ is Cohen–Macaulay and reduced}  } \\
	{ \Tilde{R}(n-1)/(\sum_{i=1}^{n-1}u_{i}v_{i},t) \text{ is Cohen–Macaulay and reduced}} \\
	{\Tilde{R}(n-1)/(t-1) \text{ is Cohen–Macaulay and reduced}} \\
	{R_1(n) \text{ is Cohen–Macaulay and reduced}} \\
	{R(n) \text{ is Cohen–Macaulay and normal}}
	\arrow["\text{ Lemma } \ref{Tilde(R)/(vu,t) CM}\text{+ Lemma } \ref{lemma: dim(Tilde(R)/(vu,t))}", Rightarrow, from=1-1, to=2-1]
	\arrow["\text{ Lemma } \ref{Tilde(R)/(t-1)}", Rightarrow, from=2-1, to=3-1]
	\arrow["\text{ Lemma } \ref{R_1 CM}", Rightarrow, from=3-1, to=4-1]
	\arrow["\text{ Lemma } \ref{R(n) CM}", Rightarrow, from=4-1, to=5-1]
\end{tikzcd}\]

The proofs of Lemma \ref{Tilde(R)/(vu,t) CM}, Lemma \ref{lemma: dim(Tilde(R)/(vu,t))} and Lemma \ref{lemma: dim(Tilde(R)/(t))} are long and technical, so we put them in Section \ref{proof of lemma Tilde(R)/(vu,t) CM}, Section \ref{Proof of lemma: dim(Tilde(R)/(vu,t))} and Section \ref{Proof of lemma: dim(Tilde(R)/(t))}, respectively. 


In order to prove Lemma \ref{Tilde(R)/(vu,t) CM}, we consider the localization of rings $R(n)$ and $R_2(n)$. Let $\mathfrak{m}'$ be the ideal of $R(n)$ generated by all of the entries of the matrices 
$X_{n}-\mathrm{diag}(0,\dots,0,1)$ and $Y_{n}$.
Let $\mathfrak{q}$ be the ideal of $R_2(n)$ generated by all of the entries of the matrices 
$X_{n}-\mathrm{diag}(0,\dots,0,1),Y_{n}$ and $u_i,v_j,v_i',t_1,t_3,t_2-1$ for $1\leq i\leq n$ and $1\leq j\leq n-1$.
Let $\mathfrak{n}=\mathfrak{q}+(v_n)$ and $w_n=\sum_{i=1}^nu_iv_i$. 
We will show that 
 \[
R_2(n)_{\mathfrak{q}}/(v_i,x_{in},x_{nn}-1,v_n'\mid 1\leq i\leq  n-1)\cong\left(\Tilde{R}(n-1)/(w_{n-1},t)\right)_{\mathfrak{q}'}\otimes_{\mathbb{K}}\mathbb{K}(v_n),
 \]
 where
 \begin{equation*}
     \mathfrak{q}'=(x_{ij},y_{ij},u_i,v_i,t_1,t_2,\mid 1\leq i,j\leq n-1)\subseteq \Tilde{R}(n-1)/(w_{n-1},t),
 \end{equation*}
 hence we establish a connection between $R_2(n)_{\mathfrak{n}}$ and $\Tilde{R}(n-1)/(w_{n-1},t)$.
Then
\[\begin{tikzcd}
	{R(n-1) \text{ is Cohen–Macaulay and reduced}  } \\
	{ R(n)_{\mathfrak{m}'} \text{ is Cohen–Macaulay} } \\
{R_2(n)_{\mathfrak{n}}/(v_1,\dots,v_n) \text{ is Cohen–Macaulay and reduced}}\\
	{R_2(n)_{\mathfrak{n}} \text{ is Cohen–Macaulay}} \\
	{ \Tilde{R}(n-1)/(w_{n-1},t) \text{ is Cohen–Macaulay}} 
	\arrow["\text{ Lemma } \ref{R(n)m' CM}", Rightarrow, from=1-1, to=2-1]
    \arrow["\text{ Lemma } \ref{R_n/(v) CM}", Rightarrow, from=2-1, to=3-1]
	\arrow["\text{ Lemma } \ref{R_n CM}", Rightarrow, from=3-1, to=4-1]
	\arrow["\text{ Lemma } \ref{proof of Tilde(R)/(vu,t) CM}", Rightarrow, from=4-1, to=5-1]
\end{tikzcd}\]

In Section \ref{Proof of lemma: dim(Tilde(R)/(vu,t))}, we divide our proof of Lemma \ref{lemma: dim(Tilde(R)/(vu,t))} in three steps.

Consider $\cCV(m)$, the set of closed points of $\Spec\Tilde{R}(m)/(t,w_{m})$, for any $m\geq 1$. We first find a subset $\cCV(m)^{\circ}\subseteq \cCV(m)$, which has an explicit irreducible decomposition. We prove that each of the irreducible components has dimension $m^{2}+m+2$ in \Cref{dim(im(psi))}.

Next, in \Cref{CV^o dense}, we prove that $\cCV(m)^{\circ}$ is dense in $\cCV(m)$. The proof is based on a case-by-case analysis, which is the content of \Cref{closure case 1: two different eigenvalues}-\ref{CV^o dense}. Hence, $\cCV(m)$ is equidimensional of $m^{2}+m+2$ and so is $\Spec\Tilde{R}(m)/(t,w_{m})$.

Finally, in Lemma \ref{Tilde(R)/(vu,t) regular point}, we show that $\Tilde{R}(m)/(t,w_m)$ is regular at all generic points by Jacobian criterion. Putting together with the result that the scheme $\Spec \Tilde{R}(m)/(t,w_{m})$ is equidimensional of $m^{2}+m+2$, we obtain the proof of Lemma \ref{lemma: dim(Tilde(R)/(vu,t))}.

In Section \ref{Proof of lemma: dim(Tilde(R)/(t))}, we will divide the proof of Lemma \ref{lemma: dim(Tilde(R)/(t))} into several lemmas. 
First, we will prove that $\mathrm{Spec}\Tilde{R}(m)[t^{-1}]/(w_m)$ is dense in $\mathrm{Spec}\Tilde{R}(m)/(w_m)$ by proving that $\mathcal{CV}(m)^o$ is a subset of the Zariski closure of $\mathrm{Spec}\Tilde{R}(m)[t^{-1}]/(w_m)$. 
Next, we will prove that $\mathrm{Spec}\Tilde{R}(m)/(w_m)$ is equidimensional of dimension $m^2+m+3$ by obtaining the upper and lower bounds of the dimension of an irreducible component by comparing them with the dimension of $R(m+1)$ and that of $\Tilde{R}(m)/(w_m,t)$.
Then, we will prove that $\mathrm{Spec}\Tilde{R}(m)/(t-1)$ is equidimensional and is regular at all generic points by proving that so is $\mathrm{Spec}\Tilde{R}(m)/(w_m)$.
Finally, we prove that $\mathrm{Spec}\Tilde{R}(m)/(t)$ is equidimensional and is regular at all generic points, by verifying those properties for $\Tilde{R}(m)[w_m^{-1}]/(t)$ and $\Tilde{R}(m)/(w_m,t)$.

\subsection*{Acknowledgments}
The first author is supported by grants from Beijing Institute of Mathematical Sciences and Applications (BIMSA), the Beijing NSF BJNSF-IS24005, and the China National Science Foundation (NSFC) NSFC-RFIS program W2432008. He would like to thank China's National Program of Overseas High Level Talent for generous support. The third author is supported by grants from BIMSA and the Beijing NSF.

\section{Notations and preliminaries}
In this section we fix some notations and record the useful lemma that will be used frequently in the subsequent sections.

Throughout the paper, $\mathbb{K}$ is an algebraically closed field. All rings are commutative, unless otherwise specified.

We denote by $M_{n\times m}(\mathbb{K})$ the set of $n\times m$ matrices over $\mathbb{K}$, by $M_{n}(\mathbb{K})$ the set of $n\times n$ matrices over $\mathbb{K}$, and by $I_n$ the identity matrix, and for a matrix $M$, we denote by $M^t$ its transpose.
For $M_i\in M_{n_i}(\mathbb{K})$ and $1\leq i\leq r$, we denote by $\mathrm{diag}(M_1,\dots,M_r)$ the block matrix $(Q_{ij})_{1\leq i,j\leq r}\in M_{n_1+\dots+n_r}(\mathbb{K})$ where $Q_{ii}=M_i$ and $Q_{ij}=0$ for $1\leq i\neq j\leq r$.
For a variety or scheme $\mathcal{X}$ over $\mathbb{K}$ and a subset $\mathcal{Y}$, we denote by $\overline{\mathcal{Y}}$ the Zariski closure of $\mathcal{Y}$ in $\mathcal{X}$.
For a scheme $\mathcal{X}$ over $\mathbb{K}$, we denote by $\mathcal{O}_{\mathcal{X}}$ the structure sheaf, by $\mathcal{O}_{\mathcal{X},p}$ the stalk of $\mathcal{O}_{\mathcal{X}}$ at $p\in \mathcal{X}$, by $\mathcal{X}(\mathbb{K})$ the set of morphisms of $\mathbb{K}$-schemes from $\mathrm{Spec}\mathbb{K}$ to $\mathcal{X}$. 
We denote by $GL_n$ the general linear group over $\mathbb{K}$, and by $\mathbb{G}_m$ the multiplicative group over $\mathbb{K}$.

The following five results are well-known, and they will be used multiple times in our proof.

\begin{lemma}[Theorem 17.3 in \cite{Matsumura} and Prop. 3.4.6 in \cite{Grothendieck}]\label{lemma: A/t CM to A CM}
    For every noetherian local ring $A$ and for every nonzerodivisor $t$ in its maximal ideal, if $A/tA$ is Cohen–Macaulay, then $A$ is Cohen–Macaulay, if $A/tA$ is reduced, then $A$ is reduced.
\end{lemma}

\begin{lemma}[Corollary 2.2.15 in \cite{Bruns-Herzog}]\label{lemma: graded CM}
Let $R=\mathbb{K}[x_1,\dots,x_n]$, $\mathfrak{m}=(x_1,\dots,x_n)$, and $M$ a finite graded $R$-module. Then $M$ is Cohen–Macaulay if and only if $M_{\mathfrak{m}}$ is Cohen–Macaulay. 
\end{lemma}

\begin{lemma}[Theorem 13.7 in \cite{Matsumura}]\label{lemma: prime divisor homogeneous}
Let $A=\bigoplus_{i\geq 0}A_i$ be a Noetherian graded ring. If $I$ is a homogeneous ideal and $P$ is a prime divisor of $I$ then $P$ is also homogeneous.    
\end{lemma}

\begin{lemma}[Proposition 1.65 in \cite{Milne}]\label{lemma:G_equivariant_implies_flat}
    Let $G$ be a group functor over $\mathbb{K}$. Let $X$ and $Y$ be nonempty algebraic schemes over $\mathbb{K}$ on which $G$ acts, and let $f:X\to Y$ be an equivariant map. If $Y$ is reduced and $G(\mathbb{K})$ acts transitively on $Y(\mathbb{K})$, then $f$ is faithfully flat.
\end{lemma}

\begin{lemma}[Serre's criterion]\label{lemma:Serres_criterion}
Let $T$ be a Noetherian ring. Consider the following conditions ($R_{i}$) and ($S_{i}$) on $T$:
\begin{itemize}
    \item [($R_{i}$)] $T_{\mathfrak{p}}$ is regular for all $\mathfrak{p}\in\mathrm{Spec}(T)$ with $\dim T_{\mathfrak{p}}\leq i$;
    \item [($S_{i}$)] $\mathrm{depth} T_{\mathfrak{p}}\geq \min\{i,\dim T_{\mathfrak{p}}\}$ for all $\mathfrak{p}\in\mathrm{Spec}(T)$.
\end{itemize}
We have that
\begin{enumerate}
    \item $T$ is reduced if and only if ($R_{0}$) and ($S_{1}$) hold, and
    \item $T$ is normal if and only if ($R_{1}$) and ($S_{2}$) hold.
\end{enumerate}
\end{lemma}
\begin{proof}
    See Theorem 23.8 in \cite{Matsumura}.
\end{proof}
\begin{remark}
    By definition, a Noetherian ring $T$ is Cohen-Macaulay if and only if it satisfies ($S_{i}$) for all $i\geq 0$.
\end{remark}
We will need the following lemma as a criterion for flatness.

\begin{lemma}\label{lemma of flat}
Let $X$ and $Y$ be $\mathbb{K}$-schemes of finite type, and assume that $X$ is equidimensional and $Y$ is normal. Let $f:X \to Y$ be a dominant morphism of finite type. Let $y$ be a closed point of $Y$ and $f|_{f^{-1}(U)}:f^{-1}(U)\to U$ be flat where $U=Y\setminus \{y\}$. Assume that the fiber $X_y=f^{-1}(y)$ is reduced and $\dim f^{-1}(y)=\dim (X)-\dim(Y)<\dim(X)$. Then $f$ is flat.

\end{lemma}
\begin{proof}
Note that $Y$ is reduced and Noetherian, and $f$ is of finite type. By Görtz-Wedhorn \cite[Theorem 14.32]{Gortz},  $f$ is flat if and only if for every discrete valuation ring $R$ and every morphism $\mathrm{Spec} R \to Y$ the pull-back morphism $X\times_Y \mathrm{Spec} R \to \mathrm{Spec} R$ is flat.

Let $R$ be a discrete valuation ring,  $r:\mathrm{Spec} R \to Y$ and $f':X_R=X\times_Y \mathrm{Spec} R \to \mathrm{Spec} R$ . If $\mathrm{im}(r)\subseteq U$, then $r$ factors through $U$, then $f'$ is flat since $f|_{f^{-1}(U)}:f^{-1}(U)\to U$ is flat. 

If $y\in \mathrm{im}(r)$, then $r(s)=y$ where $s$ is the special point of $\mathrm{Spec} R $. 
Let $\eta$ be the generic point of $Y$. 
Let $Z$ be an irreducible component of $X_R$ and $Z'$ be the irreducible component of $X$ such that $Z'$ contains the image of $Z$. 
Let $f'':Z'\to Y$ be  the compositions of the closed immersion $h:Z'\to X$ with $f$.
We have the following commutative diagram.
\[\begin{tikzcd}
Z'\times_X Z \arrow[d, hook] \arrow[r, hook]          & Z \arrow[d, hook] \arrow[rd]  &                               \\
Z'_R=Z'\times_X X_R \arrow[d] \arrow[r, "h'", hook]   & X_R \arrow[r, "f'"] \arrow[d] & \mathrm{Spec}R \arrow[d, "r"] \\
Z' \arrow[r, "h", hook] \arrow[rr, "f''", bend right] & X \arrow[r, "f"]              & Y                             \\
{}                                                    &                               &                              
\end{tikzcd}\]
Since $X$ is equidimensional and  $\dim f^{-1}(y)<\dim(X)$, we have
$\dim(Z')=\dim(X)>\dim f^{-1}(y)$, then  $Z'\cap f^{-1}(U)\neq \emptyset$ and $f''$ is  dominant.  
Since  $f|_{f^{-1}(U)}$ is flat, we have $\dim f^{-1}(p) =\dim(X)-\dim(Y)=\dim f^{-1}(\eta)=\dim f''^{-1}(\eta)$ for any  point $p \in f(U)$.
By Görtz-Wedhorn \cite[Theorem 14.110]{Gortz}, the subset 
\[
\mathcal{A}=\{z \in Z'\mid \dim_z f''^{-1}(f''(z)) = \dim f''^{-1}(\eta)\}
\]
is locally closed in $Z'$, 
where 
$\dim_z f''^{-1}(f''(z)) =\inf\dim U'$ and $U'\subseteq f''^{-1}(f''(z))$ runs through all open neighborhoods of $z$. 
For any closed point $z \in Z'$, $z$ is a closed point of $f''^{-1}(f''(z))$, then $\dim_z f''^{-1}(f''(z))=\dim(\mathcal{O}_{f''^{-1}(f''(z)),z})$ by Görtz-Wedhorn \cite[Lemma 14.94]{Gortz}, then 
\[
\dim_z f''^{-1}(f''(z)) \geq \dim(\mathcal{O}_{Z',z})-\dim(\mathcal{O}_{Y,f''(z)})=\dim f''^{-1}(\eta).
\]
On the other hand, $\dim_z f''^{-1}(f''(z)) \leq \dim f''^{-1}(\eta)$
since $f''^{-1}(f''(z))\subseteq f^{-1}(f''(z))$ and $\dim f^{-1}(f''(z))=\dim f^{-1}(\eta)=\dim f''^{-1}(\eta)$. So $z\in \mathcal{A}$ for any closed point $z \in Z'$. 
Then $\mathcal{A}=Z'$.
By Görtz-Wedhorn [\ref{Gortz}, Theorem 14.129], $f''$ is universally open, then $f'\circ h'$ is open.
Take an open subset $D$ of $X_R$ such that $D\subseteq Z$. Then $f'(D)$ is open in $\mathrm{Spec}R$ since $f'\circ h'$ is open and $h'^{-1}(D)$ is open in $Z'_R$, then the generic point $\eta'$ of $\mathrm{Spec}R$ is in $f'(D)\subseteq f'(Z)$. So the special fiber $f'^{-1}(s)$ is a subset of the closure of $f'^{-1}(\eta')$.
Since the residue field of $Y$ at $y$ is algebraically closed and $X_y$ is reduced, the fiber $X_y$ is geometrically reduced, hence $f'^{-1}(s)$  is  reduced. By Görtz-Wedhorn \cite[Proposition 14.16]{Gortz}, $f'$ is flat.
\end{proof}
The following lemmas establish connections between the ring $R(n)$ and those intermediate rings that we defined in Section \ref{sec:main_result_and_strategy}.

\begin{lemma}\label{lemma:connection_among_rings}
    We have an isomorphism
    \begin{equation*}
        R(n)/(x_{in},y_{ni}\mid 1\leq i\leq n-1)\cong\Tilde{R}(n-1)/(t-1),
    \end{equation*}
    and hence
    \begin{equation*}
        R_1(n)/(y_{ni}\mid 1\leq i\leq n-1)\cong \Tilde{R}(n-1)/(t-1).
    \end{equation*}
\end{lemma}
\begin{proof}
Consider the $\mathbb{K}$-algebra homomorphism
\begin{align*}
    \varphi:&R(n)\to\Tilde{R}(n-1)/(t-1),\\
    &x_{ij}\mapsto x_{ij},y_{ij}\mapsto y_{ij},y_{in}\mapsto u_i,x_{ni}\mapsto v_i,\\
    &x_{nn}\mapsto t_1,y_{nn}\mapsto t_2,x_{in}\mapsto 0,y_{ni}\mapsto 0,~\mbox{for}~1\leq i,j\leq n-1,
\end{align*}
which is surjecive and has kernel $(x_{in},y_{ni}\mid 1\leq i\leq n-1)$, hence
\[
R(n)/(x_{in},y_{ni}\mid 1\leq i\leq n-1)\cong \Tilde{R}(n-1)/(t-1).
\]
The second equation is given by the definition $R_1(n)=R(n)/(x_{in}\mid 1\leq i\leq n-1)$.
\end{proof}
\begin{lemma}\label{lemma:connection_among_rings2}
    Let $f_{0}=\det(X_{n-1}-x_{nn}I_{n-1})$. Then
    \begin{equation*}
        R_1(n)[f_{0}^{-1}]/(x_{ni}\mid 1\leq i\leq n-1)\cong R(n-1)[f_{0}^{-1}][x_{nn},y_{nn}].
    \end{equation*}
\end{lemma}
\begin{proof}
    Consider the homomorphism
    \begin{align*}
        \varphi:&R_1(n)[f_{0}^{-1}]\to R(n-1)[f_{0}^{-1}][x_{nn},y_{nn}],\\
        &x_{ij}\mapsto x_{ij},y_{ij}\mapsto y_{ij},x_{ni}\mapsto 0,y_{in}\mapsto 0,\\&y_{ni}\mapsto 0,x_{nn}\mapsto x_{nn},y_{nn}\mapsto y_{nn},
    \end{align*}
where $1\leq i,j\leq n-1$. One may verify that $\varphi$ is well-defined, surjective and $\ker\varphi = (x_{ni}\mid 1\leq i\leq n-1)$.
\end{proof}
We also need the following results to perform the dimension counting. They will work together with \Cref{lemma of flat} in the proof.
\begin{lemma}\label{lemma:Motzkin-Taussky}
    The scheme $\mathrm{Spec} R(n)$ is irreducible and has dimension $n^{2}+n$.
\end{lemma}
\begin{proof}
    By \cite[Theorem 6]{Motzkin-Taussky}, the scheme $\mathrm{Spec} R(n)$ is irreducible. The set of $(A,B)\in M_n(\mathbb{K})^2$ such that $A$ has only simple characteristic roots and $AB=BA$, is dense in the set of closed points of $\mathrm{Spec} R(n)$, and it has dimension $n^2+n$.
\end{proof}

\begin{lemma}\label{lemma: dim(Tilde(R)/(vu,t))}
 The scheme  $\mathrm{Spec}\Tilde{R}(m)/(w_{m},t)$ is equidimensonal of dimension $m^2+m+2$ and is regular at all generic points, where $m\geq 1$.
\end{lemma}
\begin{proof}
See Section \ref{Proof of lemma: dim(Tilde(R)/(vu,t))}.
\end{proof}

\begin{lemma}\label{lemma: dim(Tilde(R)/(t))}
 For $m\geq 1$, the scheme $\mathrm{Spec}\Tilde{R}(m)[t^{-1}]/(w_{m})$ is equidimensional of dimension $m^2+m+3$ and dense in $\mathrm{Spec}\Tilde{R}(m)/(w_{m})$. And the schemes  $\mathrm{Spec}\Tilde{R}(m)/(t)$ and $\mathrm{Spec}\Tilde{R}(m)/(t-1)$ are equidimensional of dimension $m^2+m+2$ and are regular at all generic points.
\end{lemma}
\begin{proof}
See Section \ref{Proof of lemma: dim(Tilde(R)/(t))}.
\end{proof}

\begin{corollary}\label{cor:dim_counting_formula}
The ring $R_1(n)$ is equidimensional of dimension $n^2+1$. 
\end{corollary}
\begin{proof}
By Lemma \ref{lemma:Motzkin-Taussky}, $\mathrm{Spec} R(n)$ is irreducible and has dimension $n^{2}+n$. Since the ideal $(x_{in}\mid 1\leq i\leq n-1)$ is generated by $(n-1)$ elements, for each minimal prime ideal $\mathfrak{p}$ of $R_1(n)$ we have 
\begin{equation*}
    \dim(R_1(n)/\mathfrak{p})\geq n^{2}+n-(n-1)=n^{2}+1.
\end{equation*}
On the other hand, it is easy to see that $R_1(n)$ is a graded $\mathbb{K}$-algebra with the irrelevant ideal 
    \begin{equation*}
        \mathfrak{m}_{1}=(x_{ij},y_{ij}\mid 1\leq i,j\leq n)\subseteq R_1(n).
    \end{equation*}
By \Cref{lemma:connection_among_rings}, we have that
\begin{equation*}
        R_1(n)/(y_{ni}\mid 1\leq i\leq n-1)\cong \Tilde{R}(n-1)/(t-1).
    \end{equation*}

Consider the natural map $\varphi': \mathbb{K}[y_{ni}\mid 1\leq i\leq n-1]\to R_1(n)$ and let 
\begin{equation*}
    \mathfrak{m}_2=\varphi'^{-1}(\mathfrak{m}_1)=(y_{ni}\mid 1\leq i\leq n-1).
\end{equation*}
Therefore, we have
\[
(R_1(n))_{\mathfrak{m}_1}/(\mathfrak{m}_2R_1(n))_{\mathfrak{m}_1}\cong \left(\Tilde{R}(n-1)/(t-1)\right)_{\mathfrak{m}_1'},
\]
where $\mathfrak{m}_1'=\mathfrak{m}_1\Tilde{R}(n-1)/(t-1)$.
By Lemma \ref{lemma: dim(Tilde(R)/(t))}, the local ring $(\Tilde{R}(n-1)/(t-1))_{\mathfrak{m}_1'}$ has dimension $n^2-n+2$, hence
\[
\dim (R_1(n))_{\mathfrak{m}_1}\leq n-1+\dim (\Tilde{R}(n-1)/(t-1))_{\mathfrak{m}_1'}=n^2+1.
\]
By Lemma \ref{lemma: prime divisor homogeneous}, minimal prime ideal $\mathfrak{p}$ of $R_1(n)$ is homogeneous, thus it is contained in $\mathfrak{m}_1$, hence 
\begin{equation*}
    \dim(R_1(n)/\mathfrak{p})\leq \dim (R_1(n))_{\mathfrak{m}_1}\leq n^2+1.
\end{equation*}
Hence $\dim(R_1(n)/\mathfrak{p})= n^2+1$ and therefore $R_1(n)$ is equidimensional of dimension $n^2+1$
\end{proof}

\begin{corollary}\label{cor:Rprime_eq_dim}
    The ring $R'(n)/(v_1,\dots,v_n)$ is equidimensional of dimension $n^2+n+3$. 
\end{corollary}
\begin{proof}
Since
    \[
    R'(n)/(v_1,\dots,v_n)\cong \left(\Tilde{R}(n)/(t)\right)[t_2'],
    \]
and by Lemma \ref{lemma: dim(Tilde(R)/(t))}, we see that 
$\Tilde{R}(n)/(t)$ is equidimensional of dimension $n^2+n+2$, thus
$R'(n)/(v_1,\dots,v_n)$ is equidimensional of dimension $n^2+n+3$.
\end{proof}
\section{Proof of \Cref{R(n)}}\label{result of GL}
We prove \Cref{R(n)} by induction on $n$. 
Note that for $n=1$,
\begin{equation*}
    R(1)=\mathbb{K}[x_{11},y_{11}],
\end{equation*}
which is Cohen-Macaulay and normal. 

Thus, we only need to finish the induction step. We assume that $R(n-1)$ is Cohen-Macaulay and normal, and we prove that so is $R(n)$. In order to do this, we construct some intermediate rings. One of them is
\begin{equation*}
    \Tilde{R}(n)=\mathbb{K}[x_{ij},y_{ij},u_i,v_i,t_1,t_2,t\mid 1\leq i,j\leq n]/J(n),
\end{equation*} 
where $J(n)$ is the ideal generated by all of the entries of the matrices 
\[
[X_n,Y_n]-t\begin{pmatrix}
    u_1\\
    \vdots\\
    u_n
\end{pmatrix}\left(v_1,\dots,v_n\right),
(X_n-t_{1}I_n)\begin{pmatrix}
    u_1\\
    \vdots\\
    u_n
\end{pmatrix},\left(v_1,\dots,v_n\right)(Y_n-t_{2}I_n),
\]
where $I_n$ is the identity matrix. 

\begin{lemma}\label{Tilde(R)/(vu,t) CM}
If $R(n-1)$ is Cohen–Macaulay and reduced, then $\Tilde{R}(n-1)/(\sum_{i=1}^{n-1}u_{i}v_{i},t)$ is Cohen–Macaulay. 
\end{lemma}
\begin{proof}
See Section \ref{proof of lemma Tilde(R)/(vu,t) CM}. 
\end{proof}

\begin{lemma}\label{Tilde(R)/(t-1)}
Assume that $\Tilde{R}(n-1)/(\sum_{i=1}^{n-1}u_{i}v_{i},t)$ is Cohen-Macualay, reduced and that it has dimension $n^{2}-n+2$. Then the ring $\Tilde{R}(n-1)/(t-1)$ is Cohen–Macaulay and  reduced and has dimension $n^2-n+2$.
\end{lemma}
\begin{proof}
Let
\begin{align*}
    S&=\Tilde{R}(n-1)/(\sum_{i=1}^{n-1} u_{i}v_{i})\\&\simeq\mathbb{K}[x_{ij},y_{ij},u_{i},v_{i},t_{1},t_{2},t\mid 1\leq i,j\leq n-1]/\left(J(n)+\sum_{i=1}^{n-1} u_{i}v_{i}\right).
\end{align*}
Let $\mathfrak{m}_2$ be the ideal of $S$ defined by
\begin{equation*}
    \mathfrak{m}_2=(x_{ij},y_{ij},u_i,v_i,t_1,t_2,t\mid 1\leq i,j\leq n-1).
\end{equation*} 
Let 
\[
\phi_2:\mathrm{Spec}S_{\mathfrak{m}_2}\to \mathrm{Spec}\mathbb{K}[t]_{(t)}
\]
 be the natural map. 
 Note that $\mathbb{K}[t]_{(t)}$ is a discrete valuation ring.
 According to the assumption, the special fiber $\mathrm{Spec}\left(S_{\mathfrak{m}_2}/(t)\right)$ is Cohen–Macaulay and is reduced and has dimension $n^2-n+2$.
Because $\mathrm{Spec}\left(S[t^{-1}]\right)$ is equidimensional of dimension $n^2-n+3$ and dense in $\mathrm{Spec}S$ by Lemma \ref{lemma: dim(Tilde(R)/(t))}, 
 all closed points of the special fiber of $\phi_2$ lie in the closure of generic fiber of $\phi_2$. 
 Then $\phi_2$ is flat by Görtz-Wedhorn \cite[Proposition 14.16]{Gortz}.
So $t$ is not a zero divisor of $S_{\mathfrak{m}_2}$.
Hence $S_{\mathfrak{m}_2}$ is Cohen–Macaulay and reduced by Lemma \ref{lemma: A/t CM to A CM}.

Let $\mathfrak{p}=(x_{ij},y_{ij},u_i,v_i,t_1,t_2\mid 1\leq i,j\leq n-1)$. Then $\mathfrak{p}$ is a prime ideal of $S$ and $\mathfrak{p}\subseteq\mathfrak{m}_2$. 
So $S_{\mathfrak{p}}$ is Cohen–Macaulay and is reduced and has dimension $n^2-n+2$.
Let
\[
\varphi:S[t^{-1}]\to \Tilde{R}(n-1)/(t-1)[t',t'^{-1}]
\]
\[
x_{ij}\mapsto x_{ij},y_{ij}\mapsto y_{ij},u_i\mapsto t'^{-1}u_i,v_i\mapsto v_i,t_1\mapsto t_1,t_2\mapsto t_2,t\mapsto t'
\]
be the $\mathbb{K}$-algebra homomorphism, where $1\leq i,j\leq n-1$. One may verify that $\varphi$ is well-defined and is an isomorphism.
Since $t\notin \mathfrak{p}$, we see that
\begin{align*}
S_{\mathfrak{p}}&\cong \left(\Tilde{R}(n-1)/(t-1)[t',t'^{-1}]\right)_{\varphi(\mathfrak{p})}\\
&\cong\left(\Tilde{R}(n-1)/(t-1)\right)_{\mathfrak{m}_2'}\otimes_{\mathbb{K}}\mathbb{K}(t'), 
\end{align*}
where
\begin{equation*}
    \mathfrak{m}_2'=(x_{ij},y_{ij},u_i,v_i,t_1,t_2\mid 1\leq i,j\leq n-1)\subseteq \Tilde{R}(n-1)/(t-1).
\end{equation*}
Hence $(\Tilde{R}(n-1)/(t-1))_{\mathfrak{m}_2'}$ is Cohen–Macaulay and is reduced and has dimension $n^2-n+2$.
Because $\Tilde{R}(n-1)/(t-1)$ is a finitely generated graded $\mathbb{K}[x_{ij},y_{ij},u_{i},v_i,t_1,t_2\mid 1\leq i,j\leq n-1]$-module, we see that $\Tilde{R}(n-1)/(t-1)$ is Cohen–Macaulay by Lemma \ref{lemma: graded CM}. For every minimal prime ideal $\mathfrak{p}'$ of $\Tilde{R}(n-1)/(t-1)$, it is homogeneous by Lemma \ref{lemma: prime divisor homogeneous}, thus it is contained in $\mathfrak{m}_2'$, thus the local ring $(\Tilde{R}(n-1)/(t-1))_{\mathfrak{p}'}$ is reduced. Because $\Tilde{R}(n-1)/(t-1)$ is Cohen-Macaulay and $(\Tilde{R}(n-1)/(t-1))_{\mathfrak{p}'}$ is reduced for every minimal prime ideal $\mathfrak{p}'$, the ring $\Tilde{R}(n-1)/(t-1)$ is reduced.
\end{proof}
Another intermediate ring we consider is
\begin{equation*}
    R_1(n)=R(n)/(x_{in}\mid 1\leq i\leq n-1),
\end{equation*}
and so there is a natural projection $R(n)\to R_1(n)$. Let $\mathcal{X}_0=\mathrm{Spec}R(n)$ and $\mathcal{X}_1=\mathrm{Spec}R_1(n)$.
 Let 
 \[\mathcal{Z}_i=
 \begin{cases}
     \mathrm{Spec}\mathbb{K}[x_{in}\mid 1\leq i\leq n-1]& \text{if }i=0\\
     \mathrm{Spec}\mathbb{K}[y_{ni}\mid 1\leq i\leq n-1]& \text{if }i=1
 \end{cases}
 \]
 and 
\[
\phi_i:\mathcal{X}_i\to \mathcal{Z}_i
\]
be the natural map for $i=0,1$. Let $P_0$ be the closed point of $\mathcal{Z}_0$ corresponding to the ideal 
\begin{equation*}
    (x_{in}\mid 1\leq i\leq n-1)\subseteq \mathbb{K}[x_{in}\mid 1\leq i\leq n-1],
\end{equation*}
and let $P_1$ be the closed point of $\mathcal{Z}_1$ corresponding to the ideal
\begin{equation*}
    (y_{ni}\mid 1\leq i\leq n-1)\subseteq \mathbb{K}[y_{ni}\mid 1\leq i\leq n-1].
\end{equation*}
Define $U_i=\mathcal{Z}_i\setminus\{P_i\}$ for $i=0,1$. 

%

\begin{lemma}\label{phi_i flat}
    The morphsim $\phi_i|_{\phi_i^{-1}(U_i)}:\phi_i^{-1}(U_i)\to U_i$ is faithfully flat for $i=0,1$.
\end{lemma}
\begin{proof}
Let $ GL_{n-1}$ act on $\mathcal{X}_0$ according to the rule 
\[g\cdot (A,B)= (\mathrm{diag}(g,1)\,\, A \,\, \mathrm{diag}(g,1)^{-1},\mathrm{diag}(g,1)\,\, B\,\, \mathrm{diag}(g,1)^{-1})\]
for $(A,B)\in\mathcal{X}_0(\mathbb{K})$ and $ GL_{n-1}$ 
act on $\mathcal{Z}_i$ according to the rule 
\[g\cdot \alpha=
\begin{cases}
    g\alpha\text{, \   where \   } \alpha\in M_{(n-1)\times 1}(\mathbb{K}) & \text{if } i=0\\
    \alpha g^{-1}\text{, \  where \    } \alpha\in M_{1\times (n-1)}(\mathbb{K})  & \text{if } i=1\\
\end{cases}
\]
for $\alpha\in \mathcal{Z}_i(\mathbb{K})$. 
Obviously the action of $GL_{n-1}$ leaves $U_i$, $ \mathcal{X}_i$ and $\phi_i^{-1}(U_i)$ stable. 
Then $\phi_i$ is $GL_{n-1}$-equivariant. 
   Since $U_i$ is reduced and $GL_{n-1}$ acts transitively on $U_i(\mathbb{K})$, the morphsim $\phi_i|_{\phi_i^{-1}(U_i)}$ is faithfully flat by \Cref{lemma:G_equivariant_implies_flat}.
\end{proof}

\begin{lemma}\label{R_1 CM}
If $\Tilde{R}(n-1)/(t-1)$ is Cohen-Macaulay and reduced, then the ring $R_1(n)$ is Cohen–Macaulay and reduced and has
dimension $n^2+1$.
\end{lemma}
\begin{proof}
Note that the fiber $\phi_1^{-1}(P_1)$ is $\mathrm{Spec}R_1(n)/(y_{ni}\mid 1\leq i\leq n-1)$. By \Cref{lemma:connection_among_rings}, we have that
    \begin{equation*}
        R_1(n)/(y_{ni}\mid 1\leq i\leq n-1)\cong \Tilde{R}(n-1)/(t-1).
    \end{equation*}
    Hence, by the assumption, the fiber $\phi_1^{-1}(P_1)$ is Cohen–Macaulay and reduced. By  \Cref{cor:dim_counting_formula}, we see that the scheme $\mathcal{X}_1$ is equidimensional of dimension $n^2+1$. Its dimension equals $\dim \Tilde{R}(n-1)/(t-1)=n^2-n+2=\dim\mathcal{X}_1-\dim \mathcal{Z}_1$.

    Note that $\mathcal{Z}_1$ is normal. By Lemma \ref{lemma of flat} and Lemma \ref{phi_i flat}, $\phi_1$ is flat. Thus $y_{n1},\dots,y_{n,n-1}$ is a $R_1(n)$-sequence. Let $\mathfrak{m}_1=(x_{ij},y_{ij}\mid 1\leq i,j\leq n)\subseteq R_1(n)$.
    Then $(R_1(n))_{\mathfrak{m}_1}/(y_{ni}\mid 1\leq i\leq n-1)$ is Cohen–Macaulay and reduced.
    Thus $(R_1(n))_{\mathfrak{m}_1}$ is Cohen–Macaulay and is reduced by Lemma \ref{lemma: A/t CM to A CM}. 
    Because $R_1(n)$ is a finitely generated graded $\mathbb{K}$-algebra with the irrelevant ideal $\mathfrak{m}_1$, the ring $R_1(n)$ is Cohen–Macaulay by Lemma \ref{lemma: graded CM}. For every minimal prime ideal $\mathfrak{p}'$ of $R_1(n)$, it is homogeneous by Lemma \ref{lemma: prime divisor homogeneous}, thus it is contained in $\mathfrak{m}_1$, thus the local ring $(R_1(n))_{\mathfrak{p}'}$ is reduced. Because $R_1(n)$ is Cohen-Macaulay and $(R_1(n))_{\mathfrak{p}'}$ is reduced for every minimal prime ideal $\mathfrak{p}'$, the ring $R_1(n)$ is reduced by \Cref{lemma:Serres_criterion}.
\end{proof}


\begin{lemma}\label{R(n) CM}
If $R_1(n)$ is Cohen-Macaulay and reduced, then the ring $R(n)$ is Cohen–Macaulay and normal.
\end{lemma}
\begin{proof}
By \Cref{lemma:Motzkin-Taussky}, $\mathcal{X}_0$ is irreducible of dimension $n^2+n$.
Because the fiber $\phi_0^{-1}(P_0)$ is $\mathrm{Spec}R_1(n)=\mathcal{X}_1$,
so it is Cohen–Macaulay, reduced and it has dimension $n^2+1$.
By \Cref{lemma of flat} and Lemma \ref{phi_i flat}, the morphsim $\phi_{0}$ is flat. Thus $x_{1n},\dots,x_{n-1,n}$ is an $R(n)$-sequence. Let 
\begin{equation*}
    \mathfrak{m}=(x_{ij},y_{ij}\mid 1\leq i,j\leq n)\subseteq R(n).
\end{equation*}
Since $R_1(n)$ is Cohen-Macaulay and reduced, we see that $R(n)_{\mathfrak{m}}$ is Cohen-Macaulay and reduced by Lemma \ref{lemma: A/t CM to A CM}. Because $R(n)$ is a graded $\mathbb{K}$-algebra with the irrelevant ideal $\mathfrak{m}$, so $R(n)$ is Cohen-Macaulay by \Cref{lemma: graded CM}. For every minimal prime ideal $\mathfrak{p}$ of $R(n)$, it is homogeneous by \Cref{lemma: prime divisor homogeneous}, thus it is contained in $\mathfrak{m}$, and therefore the local ring $R(n)_{\mathfrak{p}}$ is reduced. Because $R(n)$ is Cohen-Macaulay and $R(n)_{\mathfrak{p}}$ is reduced for every minimal prime ideal $\mathfrak{p}$, the ring $R(n)$ is reduced. 

Finally, we will show that the singular locus of $\mathrm{Spec}R(n)$ has codimension $\geq 2$.
If $\mathrm{char}\ \mathbb{K}=0$, this is a direct consequence of Popov \cite[Corollary 1.9]{Popov}. We will prove this for $\mathrm{char}\ \mathbb{K}\geq 0$.
Let $A,B\in M_n(\mathbb{K})$ and $AB=BA$. By \cite[Theorem 1.1]{Neubauer-Saltman}, the closed point $(A,B)$ of $\mathrm{Spec}R(n)$ is a regular point if and only if 
\[
\dim_{\mathbb{K}}\{C\in M_n(\mathbb{K})\mid AC=CA,CB=BC\}=n.
\]
Note that if $A\in M_n(\mathbb{K})$ is non-derogatory, which means that its minimal polynomial is equal to its characteristic polynomial, then $AB=BA$ if and only if $B=f(A)$ for some polynomial $f$.
Thus $(A,B)$ is a regular point if $A$ or $B$ is non-derogatory.

Let $Z$ be the subvariety of $M_n(\mathbb{K})^2$ consisted of $(A',B')\in M_n(\mathbb{K})^2$ such that $A'$ and $B'$ are derogatory and $A'B'=B'A'$, and let $Z'$ be the subvariety of $M_n(\mathbb{K})^2$ consisted of $(A',B')\in M_n(\mathbb{K})^2$ such that $\mathrm{rank}(A')\leq n-2$ and $\mathrm{rank}(B')\leq n-2$ and $A'B'=B'A'$. 
Thus $\mathrm{Spec}R(n)\setminus \overline{Z}$ is regular, and we have a surjective morphsim
\[
Z'\times \mathbb{K}^2\to Z,(A',B',a,b)\to (A'+aI_n,B'+bI_n).
\]
Thus $\dim(Z)\leq \dim(Z')+2$.
By [\ref{Basili}, Theorem 1.4], we see
\[
\dim(Z')=\max\{n^2+n-6+3t-t^2\mid 0\leq t\leq \min\{n-2,2\}\},
\]
thus $\dim(Z)\leq n^2+n-2$, thus the singular locus of $\mathrm{Spec}R(n)$ has codimension $\geq 2$.
Hence $R(n)$ is normal.
\end{proof}

By Lemma \ref{lemma: dim(Tilde(R)/(vu,t))}, Lemma \ref{Tilde(R)/(vu,t) CM}, \Cref{Tilde(R)/(t-1)}, \Cref{R_1 CM} and \Cref{R(n) CM}, we conclude that 
\begin{equation*}
    R(n-1)~\mbox{is Cohen-Macaulay and normal}~\implies~\mbox{so is}~R(n).
\end{equation*}
This completes the proof of \Cref{R(n)}. 

\section{Proof of Lemma \ref{Tilde(R)/(vu,t) CM}}\label{proof of lemma Tilde(R)/(vu,t) CM}

We keep the same notations as in \Cref{result of GL}. 
Our goal is to prove that 
$\Tilde{R}(n-1)/(w_{n-1},t)$ is Cohen–Macaulay, if $R(n-1)$ is Cohen–Macaulay and reduced.

We divide our proof in four steps.

First, we will prove that $R(n)_{\mathfrak{m}'}$ is Cohen–Macaulay (see Lemma \ref{R(n)m' CM}), where $\mathfrak{m}'$ is the ideal of $R(n)$ generated by all of the entries of the matrices 
$X_{n}-\mathrm{diag}(0,\dots,0,1)$ and $Y_{n}$. Let $\mathfrak{m}_1'=\mathfrak{m}'R_1(n)$.

Defined in Section \ref{result of GL}, recall that $R_1(n)=R(n)/(x_{in}\mid 1\leq i\leq n-1)$, $\mathcal{X}_0=\mathrm{Spec}R(n)$,  $\mathcal{X}_1=\mathrm{Spec}R_1(n)$, $     \mathcal{Z}_0=\mathrm{Spec}\mathbb{K}[x_{in}\mid 1\leq i\leq n-1]$, 
 and $\phi_0:\mathcal{X}_0\to \mathcal{Z}_0$ 
 is the natural map, and $P_0$ is the closed point of $\mathcal{Z}_0$ corresponding to the ideal $(x_{in}\mid 1\leq i\leq n-1)$,  and $U_0=\mathcal{Z}_0\setminus\{P_0\}$.

\begin{lemma}\label{dim(R_1m_1')}\label{R_1m_1' CM}
    If $R(n-1)$ is Cohen–Macaulay and reduced, then the local ring $(R_1(n))_{\mathfrak{m}_1'}$ is Cohen-Macaulay, reduced and equidimensional of dimension $n^2+1$. 
\end{lemma}
\begin{proof}
Let 
    \begin{equation*}
        I'_{1} = (x_{ni}\mid 1\leq i\leq n-1)\subseteq (R_1(n))_{\mathfrak{m}_1'}.
    \end{equation*}
    We first show that the local ring $(R_1(n))_{\mathfrak{m}_1'}/I'_{1}$ is Cohen–Macaulay and reduced. 

    Let $f_{0}=\det(X_{n-1}-x_{nn}I_{n-1})$. By \Cref{lemma:connection_among_rings2}, we have
    \begin{equation*}
        R_1(n)[f_{0}^{-1}]/(x_{ni}\mid 1\leq i\leq n-1)\cong R(n-1)[f_{0}^{-1}][x_{nn},y_{nn}].
    \end{equation*}
Let $\mathfrak{n}_{1}'$ be the image of $\mathfrak{m}_1'$ in above isomorphsim. Since $f_{0}\notin \mathfrak{m}_1'$, we have
\[
(R_1(n))_{\mathfrak{m}_1'}/I'_{1}\cong  \left(R(n-1)[x_{nn},y_{nn}]\right)_{\mathfrak{n}_{1}'}.
\]
Note that $\mathfrak{n}_{1}'$ is a maximal ideal of $R(n-1)[x_{nn},y_{nn}]$.
By \Cref{lemma:Motzkin-Taussky}, we have that
\begin{equation*}
    \dim(R_1(n))_{\mathfrak{m}'_{1}}/I'_{1}=\dim(R(n-1)[x_{n,n},y_{n,n}])_{\mathfrak{n}'_{1}}=n^{2}-n+2.
\end{equation*}
Since $R(n-1)$ is Cohen-Macaulay and reduced, it follows that $(R_1(n))_{\mathfrak{m}_1'}/I'_{1}$ is Cohen-Macaulay and reduced. 

Now, let $ \mathcal{Z}_1'=\mathrm{Spec}\mathbb{K}[x_{ni}\mid 1\leq i\leq n-1]$
 and $\phi_1':\mathcal{X}_1\to \mathcal{Z}_1'$ 
be the natural map, and $P_1'$ be the closed point of $\mathcal{Z}_1'$ corresponding to the ideal
\begin{equation*}
    (x_{ni}\mid 1\leq i\leq n-1)\subseteq\mathbb{K}[x_{ni}\mid 1\leq i\leq n-1],
\end{equation*}
and $U_1'=\mathcal{Z}_1'\setminus\{P_1'\}$.  
Note that ${\phi'_{1}}^{-1}(P'_{1})=\mathrm{Spec}R_1/(x_{ni}\mid 1\leq i\leq n-1)$ and that the set
\begin{equation*}
    \{p\in {\phi'_1}^{-1}(P'_1)\mid ~\mathcal{O}_{{\phi'_1}^{-1}(P'_1),p}~\mbox{is both Cohen-Macaulay and reduced.}\}
\end{equation*}
is open. This result is a remark of EGA \cite[Scholie 7.8.3]{Grothendieck}. Thus there is an open neighborhood of the point of ${\phi'_1}^{-1}(P'_1)$ corresponding to $\mathfrak{m}_{1}'$ such that it is both Cohen-Macaulay and reduced, hence there exists $f\in R_1(n)\setminus \mathfrak{m}_1'$ such that ${\phi''_{1}}^{-1}(P'_{1})$ is Cohen-Macaulay and reduced and has dimension $n^2-n+2$, where $\phi_1''=\phi'_{1}|_{\mathrm{Spec}R_{1}(n)[f^{-1}]}$.

By \Cref{cor:dim_counting_formula}, we have $R_1(n)[f^{-1}]$ is equidimensional and 
\[
\dim R_1(n)[f^{-1}]= n^{2}+1=\dim{\phi''_{1}}^{-1}(P'_{1})+\dim \mathcal{Z}_1'.
\]

We can define $GL_{n-1}$-actions on $\mathcal{X}_{1}$ and $\mathcal{Z}'_{1}$ in a natural way such that $U'_{1}$, $\mathcal{X}_{1}$ and ${\phi'_{1}}^{-1}(U_{1}')$ are stable under these actions, and that $\phi'_{1}$ is $GL_{n-1}$-equivariant. By \Cref{lemma:G_equivariant_implies_flat}, we can see that $\phi_1'|_{\phi_1'^{-1}(U_1')}:\phi_1'^{-1}(U_1')\to U_1'$ is faithfully flat, hence  $\phi_1''|_{\phi_1''^{-1}(U_1')}$ is flat.

Note that $\mathcal{Z}_1'$ is normal and $\phi_1''$ is dominant. By Lemma \ref{lemma of flat}, the morphsim $\phi_1''$ is flat. Thus  $x_{n1},\dots,x_{n,n-1}$ is a $R_1(n)[f^{-1}]$-sequence. 
    Since $(R_1(n))_{\mathfrak{m}_1'}/I_1'$ is Cohen–Macaulay and reduced, it follows that $(R_1(n))_{\mathfrak{m}_1'}$ is Cohen–Macaulay and reduced by Lemma \ref{lemma: A/t CM to A CM}. 
\end{proof}

\begin{lemma}\label{R(n)m' CM}
    If $R(n-1)$ is Cohen–Macaulay and reduced, then the local ring $R(n)_{\mathfrak{m}'}$ is Cohen–Macaulay of dimension $n^2+n$.
\end{lemma}
\begin{proof}
By \Cref{lemma:Motzkin-Taussky}, $\mathcal{X}_0$ is irreducible of dimension $n^2+n$.
Note that the fiber $\phi_0^{-1}(P_0)$ is $\mathrm{Spec}R_{1}=\mathcal{X}_1$.
Similar to the proof of Lemma \ref{R(n) CM} and Lemma \ref{R_1m_1' CM}, we can see that $R(n)_{\mathfrak{m}'}$ is Cohen–Macaulay of dimension $n^2+n$.
\end{proof}

Let $R'(n)$ and $R_2(n)$ be as defined in \Cref{sec:main_result_and_strategy}.
Let $\mathfrak{n}'$ be the ideal of $R'(n)$ generated by all of the entries of the matrices 
$X_{n}-\mathrm{diag}(0,\dots,0,1),Y_{n}$ and $u_i,v_i,v_i',t_1,t_3,t_2-1$ for $1\leq i\leq n$.
Let $\mathfrak{n}=\mathfrak{n}'R_2(n)$.

Next, we will prove that 
$R_2(n)_{\mathfrak{n}}/(v_1,\dots,v_n)$ is Cohen–Macaulay and reduced.

\begin{lemma}\label{R_n/(v) CM}
If $R(n)_{\mathfrak{m}'}$ is Cohen-Macaulay of dimension $n^{2}+n$, then the local ring $R_2(n)_{\mathfrak{n}}/(v_1,\dots,v_n)$ is Cohen–Macaulay and reduced and has dimension $n^2+n+2$.
\end{lemma}
\begin{proof}
    Let $\mathfrak{n}''=(x_{ij},y_{ij},x_{in},x_{ni},y_{in},y_{ni},x_{nn}-1,y_{nn},u_i,v_i,u_n,v_n,t_1,t_2\mid 1\leq i,j\leq n-1)\subseteq \Tilde{R}(n)/(t)$.
    Note that
\[
\Tilde{R}(n)/(t)= R(n)[u_i,v_i,t_1,t_2\mid 1\leq i\leq n]/J'
\]
where the ideal $J'$ is generated by all of the entries of the matrices $(X_n-t_1I_n)(u_1,\dots,u_n)^t$ and $(v_1,\dots,v_n)(Y_n-t_2I_n)$, thus it is generated by $2n$ many elements.

Since $R(n)_{\mathfrak{m}'}$ is Cohen–Macaulay of dimension $n^2+n$, $R(n)_{\mathfrak{m}'}[u_i,v_i,t_1,t_2\mid 1\leq i\leq n]$ is Cohen–Macaulay of dimension $n^2+3n+2$, thus $\left(\Tilde{R}(n)/(t)\right)_{\mathfrak{n}''}$ is Cohen–Macaulay of dimension $n^2+n+2$ and is reduced by Lemma \ref{lemma: dim(Tilde(R)/(t))}.
Note that
\begin{equation*}
    R_2(n)_{\mathfrak{n}}/(v_1,\dots,v_n)\cong \left(\Tilde{R}(n)[t'_2]/(t,f)\right)_{\mathfrak{n}'''},
\end{equation*}
where $f=\det(X_n-t_2'I_n)$ and
\begin{equation*}
\mathfrak{n}'''=\mathfrak{n}''\left(\Tilde{R}(n)[t'_2]/(t,f)\right)+(t_2'-1).
\end{equation*}
Since $\frac{\partial f}{\partial t_2'}\notin \mathfrak{n}'''$, we see that 
\[
\varphi: \mathrm{Spec} \left(\Tilde{R}(n)[t'_2]/(t,f)\right)_{\mathfrak{n}'''} \to \mathrm{Spec}\left(\Tilde{R}(n)/(t)\right)_{\mathfrak{n}''}
\]
is \'{e}tale by Milne \cite[Corollary 3.16]{Milne-Etale}. Since $\varphi$ is \'{e}tale, the closed fiber of $\varphi$ is Cohen-Macaulay,  and $\left(\Tilde{R}(n)/(t)\right)_{\mathfrak{n}''}$ is Cohen–Macaulay of dimension $n^2+n+2$ and is reduced, we see that $R_2(n)_{\mathfrak{n}}/(v_1,\dots,v_n)$ is Cohen–Macaulay of dimension $n^2+n+2$ by \cite[Theorem 2.1.7]{Bruns-Herzog}, and it is reduced by Milne \cite[Remark 3.18]{Milne-Etale}.
\end{proof}

Next, we will prove that 
$R_2(n)_{\mathfrak{n}}$ is Cohen–Macaulay.

\begin{lemma}\label{dim(R_n)}
The local ring $R_2(n)_{\mathfrak{n}}$ is equidimensional of dimension $n^2+n+3$.
\end{lemma}
\begin{proof}
Let $h_i\in R'(n)$ be defined by the formula $\det(X_{n}-(t_2-\lambda)I_{n})=\sum_{i=0}^nh_i\lambda^i$. 
Let $h=\det(X_{n-1}-x_{nn}I_{n-1})$. 
Let
\begin{align*}
\varphi:&\  R_2(n)[(h_1(t_1-t_{2}))^{-1}]/(v_1,\dots,v_{n-1},v_n-1)\to \\
&\left(\Tilde{R}(n-1)/(w_{n-1},t)\right)[x_{in},v_n',(h(t_1-x_{nn}))^{-1}\mid 1\leq i\leq n]
\end{align*}
be the $\mathbb{K}$-algebra homomorphsim such that
\[
X_n\mapsto \begin{pmatrix}
    X_{n-1}& \begin{pmatrix}
        x_{1n}\\
        \vdots\\
        x_{n-1,n}
    \end{pmatrix}\\
    0&x_{nn}
\end{pmatrix},
Y_n\mapsto \begin{pmatrix}
    Y_{n-1}&Y'\\
    0&t_2
\end{pmatrix},\begin{pmatrix}
    u_1\\
    \vdots\\
    u_n
\end{pmatrix}\mapsto \begin{pmatrix}
    u_1\\
    \vdots\\
    u_{n-1}\\
    0
\end{pmatrix},
\]
\[
(v_1',\dots,v_n')\mapsto (v_1,\dots,v_{n-1},v_n'),(t_1,t_2,t_3)\mapsto (t_1, x_{nn},t_2),
\]
\[
Y'=(Y_{n-1}-t_2I_{n-1})(X_{n-1}-x_{nn}I_{n-1})^{-1}\begin{pmatrix}
    x_{1n}\\
    \vdots\\
    x_{n-1,n}
\end{pmatrix}+(t_1-x_{nn})^{-1}\begin{pmatrix}
    u_{1}\\
    \vdots\\
    u_{n-1}
\end{pmatrix}.
\]
One may verify that $\varphi$ is well-defined and it is an isomorphsim.    
Thus $R_2(n)[(h_1(t_1-t_{2}))^{-1}]/(v_1,\dots,v_{n-1},v_n-1)$ is equidimensional of dimension $n^2+3$ by Lemma \ref{lemma: dim(Tilde(R)/(vu,t))}.

Let 
\[
\psi:\mathcal{U}=\mathrm{Spec}R_2(n)[(h_1(t_1-t_{2}))^{-1}] \to \mathcal{V}=\mathrm{Spec}\mathbb{K}[v_i\mid 1\leq i\leq n]
\]
 be the natural map and $U\subseteq\mathcal{V}$ such that $\mathcal{V}\setminus U$ is the closed point of $\mathcal{V}$ corresponding to the ideal $(v_1,\dots,v_n)$.
We can define $GL_{n}$-actions on $\mathcal{U}$ and $\mathcal{V}$ in a natural way such that $U$ and $\psi^{-1}(U)$ are stable under these actions, and that $\psi$ is $GL_{n}$-equivariant.
Similar to the proof of Lemma \ref{phi_i flat}, we can see that $\psi|_{\psi^{-1}(U)}:\psi^{-1}(U)\to U$ is faithfully flat and $\psi^{-1}(p)\cong\mathrm{Spec}R_2(n)[(h_1(t_1-t_{2}))^{-1}]/(v_1,\dots,v_{n-1},v_n-1)$ for any closed point $p\in U$.
Let $Z$ be an irreducible component of $\psi^{-1}(U)$.
We can take a closed point $q\in Z$ such that $\dim(\mathcal{O}_{Z,q})=\dim(\mathcal{O}_{\psi^{-1}(U),q})$.
Since $\psi|_{\psi^{-1}(U)}$ is flat and $\psi(q)$ is a closed point of $U$, we have
\[
\dim(Z)=\dim(\mathcal{O}_{\psi^{-1}(U),q})=\dim(\mathcal{O}_{\psi^{-1}(\psi(q)),q})+n= n^2+n+3.
\]
Hence $\psi^{-1}(U)$ is equidimensional of dimension $n^2+n+3$.

As matrices over $R'(n)$, we have $(v_1,\dots,v_n)(Y_n-t_{3}I_n)=0$, then $v_i\det(Y_n-t_{3}I_n)=0$ in $R'(n)$ for $1\leq i\leq n$. 
Let $J_1=(v_1,\dots,v_n)$ and $J_2=(\det(Y_n-t_{3}I_n))$ be the ideals of $\mathbb{K}[x_{ij},y_{ij},u_i,v_i,v_i',t_1,t_2,t_3\mid 1\leq i,j\leq n]$.
Thus $J_1J_2\subseteq J'(n)$ where $J'(n)$ is defined in \Cref{sec:main_result_and_strategy} and
\begin{align*}
&\sqrt{J'(n)}=\sqrt{J'(n)^2}\subseteq\sqrt{(J'(n)+J_1)(J'(n)+J_2)}\\
&=\sqrt{J'(n)^2+(J_1+J_2)J'(n)+J_1J_2}\subseteq\sqrt{J'(n)}.
\end{align*}
Hence 
\begin{align*}
    \mathrm{Spec}R'(n)[(h_1(t_1-t_{2}))^{-1}]&=\mathcal{U}\bigcup \mathrm{Spec}R'(n)[(h_1(t_1-t_{2}))^{-1}]/(v_1,\dots,v_n)\\
    &=\psi^{-1}(U)\bigcup \mathrm{Spec}R'(n)[(h_1(t_1-t_{2}))^{-1}]/(v_1,\dots,v_n).
\end{align*}
By \Cref{cor:Rprime_eq_dim}, 
 we see that $R'(n)[(h_1(t_1-t_{2}))^{-1}]/(v_1,\dots,v_n)$ is equidimensional of dimension $n^2+n+3$, thus $\mathcal{U}$ is equidimensional of dimension $n^2+n+3$. Thus $R_2(n)_{\mathfrak{n}}$ is equidimensional of dimension $n^2+n+3$ since $h_1(t_1-t_{2})\notin \mathfrak{n}$.
\end{proof}

\begin{lemma}\label{R_n CM}
If $R_2(n)_{\mathfrak{n}}/(v_1,\dots,v_n)$ is Cohen–Macaulay and reduced and has dimension $n^2+n+2$, then the local ring $R_2(n)_{\mathfrak{n}}$ is Cohen–Macaulay of dimension $n^2+n+3$.
\end{lemma}
\begin{proof}
As matrices over $R_2(n)_{\mathfrak{n}}/(v_n)$, we have $(v_1,\dots,v_{n-1},0)(X_n-t_{2}I_n)=0$, then $v_i=0$ in $R_2(n)_{\mathfrak{n}}/(v_n)$ for $1\leq i\leq n$ since $\det(X_{n-1}-t_2I_{n-1})\notin\mathfrak{n}$. Thus $R_2(n)_{\mathfrak{n}}/(v_n)=R_2(n)_{\mathfrak{n}}/(v_1,\dots,v_n)$. 
Since  $R_2(n)_{\mathfrak{n}}/(v_1,\dots,v_n)$ is Cohen–Macaulay and reduced and has dimension $n^2+n+2$, so is $R_2(n)_{\mathfrak{n}}/(v_n)$.
By Lemma \ref{dim(R_n)}, the local ring $R_2(n)_{\mathfrak{n}}$ is equidimensional of dimension $n^2+n+3$.
Let 
\[
\psi':\mathrm{Spec}R_2(n)_{\mathfrak{n}}\to \mathrm{Spec}\mathbb{K}[v_n]_{(v_n)}
\]
 be the natural map. 
 Similar to the proof of Lemma \ref{Tilde(R)/(t-1)}, we can see that $\psi'$ is flat and $R_2(n)_{\mathfrak{n}}$ is Cohen–Macaulay.
\end{proof}

Finally, we will prove that 
$\Tilde{R}(n-1)/(w_{n-1},t)$ is Cohen–Macaulay.

\begin{lemma}\label{proof of Tilde(R)/(vu,t) CM}
If $R_2(n)_{\mathfrak{n}}$ is Cohen–Macaulay of dimension $n^2+n+3$, then The ring $\Tilde{R}(n-1)/(w_{n-1},t)$ is Cohen–Macaulay.
\end{lemma}
\begin{proof}
    Let $\mathfrak{q}$ be the ideal of $R_2(n)$ generated by all of the entries of the matrices 
$X_{n}-\mathrm{diag}(0,\dots,0,1),Y_{n}$ and $u_i,v_j,v_i',t_1,t_3,t_2-1$ for $1\leq i\leq n$ and $1\leq j\leq n-1$.
Thus $\mathfrak{q}$ is a prime ideal of $R_2(n)$ and $\mathfrak{q}\subseteq\mathfrak{n}$. 
Since $R_2(n)_{\mathfrak{n}}$ is Cohen–Macaulay of dimension $n^2+n+3$, the local ring $R_2(n)_{\mathfrak{q}}$ is Cohen–Macaulay and has dimension $n^2+n+2$.
Similar to the proof of Lemma \ref{dim(R_n)}, we see that 
 \[
  R_2(n)_{\mathfrak{q}}/(v_i,x_{in},x_{nn}-1,v_n'\mid 1\leq i\leq  n-1)\cong\left(\Tilde{R}(n-1)/(w_{n-1},t)\right)_{\mathfrak{q}'}\otimes_{\mathbb{K}}\mathbb{K}(v_n),
 \]
 where $\mathfrak{q}'=(x_{ij},y_{ij},u_i,v_i,t_1,t_2,\mid 1\leq i,j\leq n-1)\subseteq \Tilde{R}(n-1)/(w_{n-1},t)$.
Since $\dim\left(\Tilde{R}(n-1)/(w_{n-1},t)\right)_{\mathfrak{q}'}=n^2-n+2=\dim(R_2(n)_{\mathfrak{q}})-2n$ by Lemma \ref{lemma: dim(Tilde(R)/(vu,t))}, we see that 
 $\left(\Tilde{R}(n-1)/(w_{n-1},t)\right)_{\mathfrak{q}'}$ is Cohen–Macaulay.
Because $\Tilde{R}(n-1)/(w_{n-1},t)$ is a finitely generated graded $\mathbb{K}$-algebra with the irrelevant ideal $\mathfrak{q}'$, the ring $\Tilde{R}(n-1)/(w_{n-1},t)$ is Cohen–Macaulay by Lemma \ref{lemma: graded CM}.
\end{proof}
 
This completes the proof of Lemma \ref{Tilde(R)/(vu,t) CM}.

\section{Proofs of Lemma \ref{lemma: dim(Tilde(R)/(vu,t))} and Lemma \ref{lemma: dim(Tilde(R)/(t))}}\label{proof of lemma dim(Tilde(R)/(vu,t)) and lemma dim(Tilde(R)/(t))}

We keep the same notations as in Section \ref{result of GL} and Section \ref{proof of lemma Tilde(R)/(vu,t) CM}.

\subsection{Proof of Lemma \ref{lemma: dim(Tilde(R)/(vu,t))}}\label{Proof of lemma: dim(Tilde(R)/(vu,t))}\ 

We divide our proof of Lemma \ref{lemma: dim(Tilde(R)/(vu,t))} in three steps.

First, we need to consider some irreducible subsets of dimension $m^2+m+2$.

For any $m\geq 1$, let $\mathcal{CV}(m)=\{(A,B,\alpha,\beta,a,b)\in M_{m}(\mathbb{K})^2\times M_{m\times 1}(\mathbb{K})\times M_{1\times m}(\mathbb{K})\times \mathbb{K}^2\mid AB=BA,A\alpha=a\alpha,\beta B=b\beta, \beta \alpha=0\}$.
 For any $i\geq 0$, let $\Lambda_i=\{(a_1,\dots,a_{i+1})\in \mathbb{K}^{i+1}\mid \forall k\neq j, a_k\neq a_j\}$. 
For $0\leq m_1,m_2\leq m$ and $m_1+m_2\leq m$, 
let 
\[
\psi_{m,m_1,m_2}:GL_m(\mathbb{K})\times \Lambda_{m-m_1} \times \Lambda_{m-m_2}\times\mathbb{K}^{m_1}\times\mathbb{K}^{m_2} \to \mathcal{CV}(m)
\]
\[
(g,a_0,\dots,a_{m-m_1},b_0,\dots,b_{m-m_2},\alpha, \beta )
\mapsto (g\ \mathrm{diag}(a_0 I_{m_1},a_1,\dots,a_{m-m_1})g^{-1},
\]\[
g\ \mathrm{diag}(b_{1},\dots,b_{m-m_2},
b_0I_{m_2})g^{-1},
g(\alpha,0)^t,(0, \beta )g^{-1},a_0,b_0)
\]
and 
\[
\mathcal{CV}(m)^o=\bigcup_{0\leq m_1,m_2,m_1+m_2\leq m}\mathrm{im}(\psi_{m,m_1,m_2}).
\]
Let $GL_{m}(\mathbb{K})$ act on $ \mathcal{CV}(m)$ according to the rule $g\cdot (A,B,\alpha, \beta ,a,b)=(gAg^{-1},gBg^{-1},g\alpha, \beta g^{-1},a,b)$ for any $(A,B,\alpha, \beta ,a,b)\in \mathcal{CV}(m)$.
Obviously the action of $GL_{m}( \mathbb{K})$ leaves $\mathcal{CV}(m)^o$ and $\overline{\mathcal{CV}(m)^o}$ stable.

\begin{lemma}\label{dim(im(psi))}
For $0\leq m_1,m_2\leq m$ and $m_1+m_2\leq m$, $\mathrm{im}(\psi_{m,m_1,m_2})$ is irreducible and has dimension $m^2+m+2$. 
\end{lemma}
\begin{proof}
    Since $GL_m(\mathbb{K})$ and $ \Lambda_{i}$ are irreducible, we see that $\mathrm{im}(\psi_{m,m_1,m_2})$ is irreducible.
    For any $(g,a_0,\dots,a_{m-m_1},b_0,\dots,b_{m-m_2},\alpha, \beta )\in GL_m(\mathbb{K})\times \Lambda_{m-m_1} \times \Lambda_{m-m_2}\times\mathbb{K}^{m_1}\times\mathbb{K}^{m_2}$, then $
(g',a_0',\dots,a_{m-m_1}',b_0',\dots,b_{m-m_2}',\alpha',
 \beta')\in \psi_{m,m_1,m_2}^{-1}(\psi_{m,m_1,m_2}(g,a_0,\dots,a_{m-m_1},b_0,\dots,b_{m-m_2},\alpha, \beta ))$
if and only if
\[
\mathrm{diag}(a_0'I_{m_1},a'_1,\dots,a'_{m-m_1})=g'^{-1}g\ \mathrm{diag}(a_0I_{m_1},a_1,\dots,a_{m-m_1})g^{-1}g',
\]
\[
\mathrm{diag}(b_1',\dots,b'_{m-m_2},b_0' I_{m_2})=
g'^{-1}g\ \mathrm{diag}(b_1,\dots,b_{m-m_2},b_0I_{m_2})g^{-1}g',
\]
\[
a_0'=a_0,b_0'=b_0,(\alpha',0)^t=g'^{-1}g(\alpha,0)^t,(0, \beta')=(0, \beta )g^{-1}g',
\]
\[
g'^{-1}g=Q\ \mathrm{diag}(d_1,\dots,d_m),d_1,\dots,d_m\in \mathbb{K}\setminus\{0\},
\]
where $Q$ is a permutation matrix.
Thus
\[
\dim\psi_{m,m_1,m_2}^{-1}(\psi_{m,m_1,m_2}(g,a_0,\dots,a_{m-m_1},b_0,\dots,b_{m-m_2},\alpha, \beta ))=m.
\]
Hence
\begin{align*}
&\dim(\mathrm{im}(\psi_{m,m_1,m_2}))\\
&=m^2+(m-m_1+1)+(m-m_2+1)+m_1+m_2-m\\
&=m^2+m+2.
\end{align*}
\end{proof}

Next, we will prove that $\mathcal{CV}(m)^o$ is dense in $\mathcal{CV}(m)$, thus we can obtain all irreducible components of $\mathrm{Spec}\Tilde{R}(m)/(w_{m},t)$. 
We will prove it by induction on $m$.  The proof is clear for $m=1$. We assume $m\geq 2$ and make the inductive hypothesis that if $0<r<m$, then $\overline{\mathcal{CV}(r)^o}=\mathcal{CV}(r)$. Let $\mathcal{CV}(m)_1$ be the set of $(A_1,A_2,\alpha,\beta,a,b)\in \mathcal{CV}(m)$ such that $A_i$ has a single eigenvalue with an eigenspace of dimension $\leq \frac{m}{2}$ for $i=1,2$.
For $(A,B,\alpha,\beta,a,b)\in \mathcal{CV}(m)$, we will prove that $(A,B,\alpha,\beta,a,b)\in \overline{\mathcal{CV}(m)^o}$ in following cases:
\begin{enumerate}
\item If $A$ or $B$ has at least two different eigenvalues, then $(A,B,\alpha,\beta,a,b)\in \overline{\mathcal{CV}(m)^o}$. (See Lemma \ref{closure case 1: two different eigenvalues}.)
\item If $m=2$, then $(A,B,\alpha,\beta,a,b)\in \overline{\mathcal{CV}(m)^o}$. (See Lemma \ref{closure of CV(2)^o}.)
\item If there is $g\in GL_m(\mathbb{K})$ such that
\[
gAg^{-1}=\begin{pmatrix}
    a'I_r& A_2\\
    0&A_4
\end{pmatrix},gBg^{-1}=\begin{pmatrix}
    B_1& B_2\\
    0&B_4
\end{pmatrix},g^{-1}\beta=(0,\beta_1),
\]
where $B_1\in M_{r}(\mathbb{K}),\beta_1 \in M_{1\times (m-r)}(\mathbb{K})$, $\mathrm{rank}(A_2)<r,0<r<m$, then $(A,B,\alpha,\beta,a,b)\in \overline{\mathcal{CV}(m)^o}$. (See Lemma \ref{closure case 3: rank(A_2)<r}.)
\item If there is $g\in GL_m(\mathbb{K})$ such that
\[
gAg^{-1}=\begin{pmatrix}
    a'I_r& A_2\\
    0&a'I_{m-r}
\end{pmatrix},gBg^{-1}=\begin{pmatrix}
    B_1& B_2\\
    0&B_4
\end{pmatrix},g^{-1}\beta=(0,\beta_1),
\]
where $B_1\in M_{r}(\mathbb{K}),\beta_1 \in M_{1\times (m-r)}(\mathbb{K})$, $\mathrm{rank}(A_2)<m-r,0<r<m$, then $(A,B,\alpha,\beta,a,b)\in \overline{\mathcal{CV}(m)^o}$. (See Lemma \ref{closure case 4: A_4=a'}.)
\item If $m\geq 3$ and $(A,B,\alpha,\beta,a,b)\in\mathcal{CV}(m)\setminus\mathcal{CV}(m)_1$, then $(A,B,\alpha,\beta,a,b)\in \overline{\mathcal{CV}(m)^o}$ by case $(1),(3),(4)$. (See Lemma \ref{CV^o dense}.)
\item If $m\geq 3$ and $(A,B,\alpha,\beta,a,b)\in\mathcal{CV}(m)_1$, then we will show that $\dim(\mathcal{CV}(m)_1)\leq m^2+m$ and $\overline{\mathcal{CV}(m)\setminus \mathcal{CV}(m)_1}=\mathcal{CV}(m)$ (see Lemma \ref{closure case 2: dim(eigenspace)<=m/2}), thus $(A,B,\alpha,\beta,a,b)\in \overline{\mathcal{CV}(m)^o}$ by case $(5)$.
\end{enumerate}

\begin{lemma}\label{closure case 1: two different eigenvalues}
Let $(A,B,\alpha, \beta ,a,b)\in \mathcal{CV}(m)$ such that  $A$ or $B$ has at least two different eigenvalues.
 If $
\overline{\mathcal{CV}(m')^o}=\mathcal{CV}(m')$ for any $m'<m$, then $
(A,B,\alpha, \beta ,a,b)\in \overline{\mathcal{CV}(m)^o}$.
\end{lemma}
\begin{proof} 
We may assume that $A$ has at least two different eigenvalues and let
\[
A=\begin{pmatrix} 
A_1& 0\\  0&A_4
\end{pmatrix},B=\begin{pmatrix} 
B_1& B_2\\ B_3&B_4
\end{pmatrix}\in M_{m}(\mathbb{K}),\alpha=\begin{pmatrix} \alpha_1\\ 0\end{pmatrix}\in M_{m\times 1}(\mathbb{K}),
\]
and $ \beta =( \beta_1, \beta_2)\in M_{1\times m}(\mathbb{K})$, such that $A_1,B_1\in M_{m'}(\mathbb{K}),\alpha_1\in M_{m'\times 1}(\mathbb{K})$, $ \beta_1 \in M_{1\times m'}(\mathbb{K}), \beta_2 \in M_{1\times (m-m')}(\mathbb{K})$ where $0<m'<m$, and $A_1$ and $A_4$ have no common eigenvalues.
Since $AB=BA$, it follows that $A_iB_i=B_iA_i$ for $i=1,4$ and $A_1B_2=B_2A_4$ and $A_4B_3=B_3A_1$, therefore $B_2=0$ and $B_3=0$ since $A_1$ and $A_4$ have no common eigenvalues.
Since $A\alpha=a\alpha$ and $\beta B=b \beta $, we have $A_1\alpha_1=a\alpha_1, \beta_1B_1=b\beta_1,\beta_2B_4=b\beta_2,\beta_1\alpha_1=0$, thus $(A_1,B_1,\alpha_1,\beta_1,a,b)\in \mathcal{CV}(m')$ and $(A_4,B_4,0,\beta_2,a,b)\in \mathcal{CV}(m-m')$.

We have a morphsim
\[
F:\mathcal{CV}(m')\times \mathcal{CV}(m-m')\to \mathcal{CV}(m),
\]
\[
(A',B',\alpha',\beta',a',b',A'',B'',\alpha'',\beta'',a'',b'')\mapsto (\mathrm{diag}(A',A''-(a''-a')I_{m-m'}),
\]
\[
\mathrm{diag}(B',B''-(b''-b')I_{m-m'}),(\alpha',\alpha'')^t,(\beta',\beta''),a',b').
\]
For $0\leq m_1,m_2\leq m',0\leq m_1',m_2'\leq m-m',$ and $m_1+m_2\leq m',m_1'+m_2'\leq m-m'$, 
it follows from Lemma \ref{dim(im(psi))} that $\mathrm{im}(\psi_{m',m_1,m_2})$ and $\mathrm{im}(\psi_{m-m',m_1',m_2'})$ are irreducible, thus $\mathrm{im}(\psi_{m',m_1,m_2})\times\mathrm{im}(\psi_{m-m',m_1',m_2'})$ is irreducible.
Let $U$ be the set of $(A_1',A_2',\alpha',\beta',a_1',a_2',A_1'',A_2'',\alpha'',\beta'',a_1'',a_2'')\in\mathrm{im}(\psi_{m',m_1,m_2})\times\mathrm{im}(\psi_{m-m',m_1',m_2'})$ such that 
\[
\{\text{eigenvalues of }A_i'\}\cap \{\text{eigenvalues of }A_i''-(a_i''-a_i')I_{m-m'}\}\subseteq \{a_i'\},i=1,2.
\]
Thus $F(U)\subseteq \mathcal{CV}(m)^o$.
Note that $U\neq \emptyset$ is open in $\mathrm{im}(\psi_{m',m_1,m_2})\times\mathrm{im}(\psi_{m-m',m_1',m_2'})$.
Thus $F(\mathrm{im}(\psi_{m',m_1,m_2})\times\mathrm{im}(\psi_{m-m',m_1',m_2'}))\subseteq \overline{\mathcal{CV}(m)^o}$.
So $F(\mathcal{CV}(m')^o\times \mathcal{CV}(m-m')^o)\subseteq \overline{\mathcal{CV}(m)^o}$.
Clearly $(A,B,\alpha,\beta,a,b)\in \mathrm{im}(F) \subseteq  \overline{F(\mathcal{CV}(m')^o\times \mathcal{CV}(m-m')^o)}\subseteq \overline{\mathcal{CV}(m)^o}$.

\end{proof}

\begin{lemma}\label{closure case 2: dim(eigenspace)<=m/2}
Let $m\geq 3$ and $\mathcal{CV}(m)_1$ be the set of $(A_1,A_2,\alpha,\beta,a,b)\in \mathcal{CV}(m)$ such that $A_i$ has a single eigenvalue with an eigenspace of dimension $\leq \frac{m}{2}$ for $i=1,2$.
 Then $\dim(\mathcal{CV}(m)_1)\leq m^2+m$ and $\overline{\mathcal{CV}(m)\setminus \mathcal{CV}(m)_1}=\mathcal{CV}(m)$.
\end{lemma}
\begin{proof}
Let $\mathcal{N}$ be the set of $(A_1,A_2)\in M_m(\mathbb{K})^2$ such that $A_1A_2=A_2A_1$ and $A_i$ has a single eigenvalue for $i=1,2$, and let $\mathcal{N}'$ be the set of $(A_1,A_2)\in M_m(\mathbb{K})^2$ such that $A_1A_2=A_2A_1$ and $A_i$ is nilpotent for $i=1,2$.  
By \cite[Theorem 3.7]{Premet}, we see that $\mathcal{N}'$ is irreducible of dimension $m^2-1$. Hence $\mathcal{N}$ is irreducible of dimension $(m^2-1)+2=m^2+1$ since $\mathcal{N}'\times \mathbb{K}^2\cong\mathcal{N}$.

Let $\pi:\mathcal{CV}(m)_1\to \mathcal{N}$ be the natural map, and let $U$ be the set of $(A_1,A_2)\in \mathcal{N}$ such that the eigenspace of $A_i$ has dimension $1$ for $i=1,2$, and $V=\mathcal{N}\setminus U$. 
Because $\mathcal{N}$ is irreducible and $U\neq \emptyset$ is open in $\mathcal{N}$, we have 
$\dim(U)=m^2+1$ and $\dim(V)\leq m^2$. 
Since $m\geq 3$, we see that $\dim(\pi^{-1}(p))\leq 2$ for any $p\in U$ and $\dim(\pi^{-1}(p'))\leq \frac{m}{2}+\frac{m}{2}=m$ for any $p'\in V$, thus $\dim(\pi^{-1}(U))\leq \dim(U)+2\leq m^2+m$ and $\dim(\pi^{-1}(V))\leq \dim(V)+m\leq m^2+m$. Since $\mathcal{CV}(m)_1=\pi^{-1}(U)\cup\pi^{-1}(V)$, we have $\dim(\mathcal{CV}(m)_1)\leq m^2+m$.
By Lemma \ref{lemma:Motzkin-Taussky}, the scheme $\mathrm{Spec}R(m)$ is irreducible of dimension $m^2+m$, thus each irreducible component of $\mathcal{CV}(m)$ has dimension $\geq m^2+m+1$.
Hence $\overline{\mathcal{CV}(m)\setminus \mathcal{CV}(m)_1}=\mathcal{CV}(m)$.
\end{proof}

\begin{lemma}\label{closure case 3: rank(A_2)<r}
Let
\[
A=\begin{pmatrix}
    a'I_r& A_2\\
    0&A_4
\end{pmatrix},B=\begin{pmatrix}
    B_1& B_2\\
    0&B_4
\end{pmatrix},\beta=(0,\beta_1),
\]
where $B_1\in M_{r}(\mathbb{K}),\beta_1 \in M_{1\times (m-r)}(\mathbb{K})$, $\mathrm{rank}(A_2)<r,0<r<m$ and $
(A,B,\alpha,\beta,a,b)\in \mathcal{CV}(m)$.
If $
\overline{\mathcal{CV}(m')^o}=\mathcal{CV}(m')$ for any $m'<m$, then $
(A,B,\alpha,\beta,a,b)\in \overline{\mathcal{CV}(m)^o}$.
\end{lemma}
\begin{proof}
If $B$ has at least two different eigenvalues, then $
(A,B,\alpha,\beta,a,b)\in \overline{\mathcal{CV}(m)^o}$ by Lemma \ref{closure case 1: two different eigenvalues}.
We may assume that $B$ has a single eigenvalue $b'$.
Since $\mathrm{rank}(A_2)<r$, there is $0\neq \beta_2\in M_{1\times r}(\mathbb{K})$ such that $\beta_2A_2=0$.
Take $\alpha_2\in M_{r\times 1}(\mathbb{K})$ such that $\beta_2\alpha_2\neq 0$.
For any $c\in \mathbb{K}$, let 
\[
B(c)=\begin{pmatrix}
    B_1+c\alpha_2\beta_2& B_2\\
    0& B_4
\end{pmatrix},
\]
then $AB(c)=B(c)A$ and $\beta B(c)=b\beta$, thus $(A,B(c),\alpha,\beta,a,b)\in \mathcal{CV}(m)$.
Let $D=\mathbb{K}\setminus\{ 0\}$. 
For any $c\in D$, we have $\mathrm{tr}(B(c))=mb'+c\beta_2\alpha_2\neq mb'$ and $\det(B(c)-b'I_m)=0$, thus
$B(c)$ has at least two different eigenvalues, thus $(A,B(c),\alpha,\beta,a,b)\in \overline{\mathcal{CV}(m)^o}$ by Lemma \ref{closure case 1: two different eigenvalues}.
Since $D$ is dense in $\mathbb{K}$, we see that $(A,B,\alpha,\beta,a,b)\in\overline{\mathcal{CV}(m)^o}$.    
\end{proof}

\begin{lemma}\label{closure case 4: A_4=a'}
    Let
\[
A=\begin{pmatrix}
    a'I_r& A_2\\
    0&a'I_{m-r}
\end{pmatrix},B=\begin{pmatrix}
    B_1& B_2\\
    0&B_4
\end{pmatrix},\beta=(0,\beta_1),
\]
where $B_1\in M_{r}(\mathbb{K}),\beta_1 \in M_{1\times (m-r)}(\mathbb{K})$, $\mathrm{rank}(A_2)<m-r,0<r<m$ and $
(A,B,\alpha,\beta,a,b)\in \mathcal{CV}(m)$.
If $
\overline{\mathcal{CV}(m')^o}=\mathcal{CV}(m')$ for any $m'<m$, then $
(A,B,\alpha,\beta,a,b)\in \overline{\mathcal{CV}(m)^o}$.
\end{lemma}
\begin{proof}
If $B$ has at least two different eigenvalues, then $
(A,B,\alpha,\beta,a,b)\in \overline{\mathcal{CV}(m)^o}$ by Lemma \ref{closure case 1: two different eigenvalues}.
We may assume that $B$ has a single eigenvalue $b'$. 
Since $\mathrm{rank}(A_2)<m-r$, there is $0\neq \alpha'\in M_{(m-r)\times 1}(\mathbb{K})$ such that $A_2\alpha'=0$. 
Take 
\[
B_4'=\begin{cases}
    \alpha'\gamma, \text{ where } \gamma\in M_{1\times (m-r)}(\mathbb{K}),\gamma\alpha'=1& \text{ if }\beta_1\alpha'=0\\
    \alpha' \beta_1& \text{ if }\beta_1\alpha'\neq 0
\end{cases}.
\]
For any $c\in \mathbb{K}$, let 
\[
B(c)=\begin{pmatrix}
    B_1& B_2\\
    0& B_4+cB_4'
\end{pmatrix},
\]
then $AB(c)=B(c)A$ and $\beta B(c)=(b+c\beta_1\alpha')\beta$, thus $(A,B(c),\alpha,\beta,a,b+\beta_1\alpha')\in \mathcal{CV}(m)$.
Let $D=\mathbb{K}\setminus\{ 0\}$. 
For any $c\in D$, we have $\mathrm{tr}(B(c))=mb'+cB_4'\neq mb'$ and $\det(B(c)-b'I_m)=0$, thus
$B(c)$ has at least two different eigenvalues, thus $(A,B(c),\alpha,\beta,a,b+c\beta_1\alpha')\in \overline{\mathcal{CV}(m)^o}$ by Lemma \ref{closure case 1: two different eigenvalues}.
Since $D$ is dense in $\mathbb{K}$, we see that $(A,B,\alpha,\beta,a,b)\in\overline{\mathcal{CV}(m)^o}$.
\end{proof}

\begin{lemma}\label{closure of CV(2)^o}
    $\overline{\mathcal{CV}(2)^o}=\mathcal{CV}(2)$.
\end{lemma}
\begin{proof}
Let $(A,B,\alpha,\beta,a,b)\in \mathcal{CV}(2)$. Note that $\overline{\mathcal{CV}(1)^o}=\mathcal{CV}(1)$. If $A$ or $B$ has at least two different eigenvalues, 
then $(A,B,\alpha,\beta,a,b)\in \overline{\mathcal{CV}(2)^o}$ by Lemma \ref{closure case 1: two different eigenvalues}.
If $A=a_1I_2$ or $B=b_1I_2$, then we may assume 
\[
A=a_1I_2,B=\begin{pmatrix}
    b_1&b_2\\
    0&b_3
\end{pmatrix},\alpha=\begin{pmatrix}
    c_1\\
    c_2
\end{pmatrix},\beta=(0,c_3),c_2c_3=0.
\]
 By Lemma \ref{closure case 4: A_4=a'}, we can see that $(A,B,\alpha,\beta,a,b)\in\overline{\mathcal{CV}(2)^o}$.

If $A$ has a single eigenvalue with an eigenspace of dimension $1$, we may assume
\[
A=\begin{pmatrix}
    a_1&a_2\\
    0&a_1
\end{pmatrix},B=\begin{pmatrix}
    b_1&b_2\\
    0&b_1
\end{pmatrix},\alpha=\begin{pmatrix}
    c\\
    0
\end{pmatrix},\beta=(c',c''),a_2\neq 0,cc'=0.
\]
If $b_2=0$, then $(A,B,\alpha,\beta,a,b)\in\overline{\mathcal{CV}(2)^o}$ by similar to the proof of Lemma \ref{closure case 4: A_4=a'}. We may assume $b_2\neq 0$, so $c'=0$.
For any $d\in \mathbb{K}$, let 
\[
A(d)=\begin{pmatrix}
    a_1+d&a_2\\
    0&a_1
\end{pmatrix},B(d)=\begin{pmatrix}
    b_1+\frac{b_2}{a_2}d&b_2\\
    0&b_1
\end{pmatrix},
\]
then $A(d)B(d)=B(d)A(d)$ and $\beta B(d)=b\beta,A(d)\alpha=(a+d)\alpha$, thus $(A(d),B(d),\alpha,\beta,a+d,b)\in \mathcal{CV}(2)$.
Let $D=\mathbb{K}\setminus\{0\}$. 
For any $d\in D$, the matrix $A(d)$ has two different eigenvalues, thus $(A(d),B(d),\alpha,\beta,a+d,b)\in \overline{\mathcal{CV}(2)^o}$ by Lemma \ref{closure case 1: two different eigenvalues}.
Since $D$ is dense in $\mathbb{K}$, we see that $(A,B,\alpha,\beta,a,b)\in\overline{\mathcal{CV}(2)^o}$.
\end{proof}

\begin{lemma}\label{CV^o dense}
    For any $m\geq 1$, $\mathcal{CV}(m)^o$ is dense in $\mathcal{CV}(m)$.
\end{lemma}
\begin{proof}
     We will prove this Lemma by induction on $m$.  The proof is clear for $m=1$ and $m=2$ by Lemma \ref{closure of CV(2)^o}. We assume $m\geq 3$ and make the inductive hypothesis that if $0<r<m$, then $\overline{\mathcal{CV}(r)^o}=\mathcal{CV}(r)$. 
Let $\mathcal{CV}(m)_1$ be the set of $(A_1,A_2,\alpha,\beta,a,b)\in \mathcal{CV}(m)$ such that $A_i$ has a single eigenvalue with an eigenspace of dimension $\leq \frac{m}{2}$ for $i=1,2$, and let $V=\mathcal{CV}(m)\setminus \mathcal{CV}(m)_1$.
Thus $\dim(\mathcal{CV}(m)_1)\leq m^2+m$ and $\overline{V}=\mathcal{CV}(m)$. by Lemma \ref{closure case 2: dim(eigenspace)<=m/2}.     
So it suffices to show that $V\subseteq \overline{\mathcal{CV}(m)^o}$.
     
Let $(A,B,\alpha,\beta,a,b)\in V$. 
If $A$ or $B$ has at least two different eigenvalues, 
then $(A,B,\alpha,\beta,a,b)\in \overline{\mathcal{CV}(m)^o}$ by Lemma \ref{closure case 1: two different eigenvalues}.
Thus we may assume that $A$ has a single eigenvalue $a'$ with an eigenspace of dimension $> \frac{m}{2}$ and $B$ has a single eigenvalue $b'$.
Let $W=\{\gamma\in M_{m\times 1}(\mathbb{K})\mid A\gamma=a'\gamma,\beta \gamma=0\}$.
Thus $\alpha\in W$ and $\dim(W)>\frac{m}{2}-1$.
For any $\gamma\in W$, since $AB=BA$, we have $
AB\gamma=a'B\gamma$ and $\beta B\gamma=b\beta\gamma=0$, thus $B\gamma\in W$.
If $\dim(W)=m$, then $A=a'I_m$ and $\beta=0$, thus there is $P'$ such that
\[
P'BP'^{-1}=\begin{pmatrix}
    B_1'&B_2'\\
0&b'
\end{pmatrix}, 
\]
thus $(A,B,\alpha,\beta,a,b)\in \overline{\mathcal{CV}(m)^o}$ by Lemma \ref{closure case 4: A_4=a'}. Hence we may assume $\dim(W)<m$.
Take $P\in GL_m(\mathbb{K})$ such that the first $r$ columns of $P$ is a $\mathbb{K}$-linear basis of $W$ and $0<\frac{m}{2}-1<r<m$.
Let $A'=P^{-1}AP,B'=P^{-1}BP,\alpha'=P^{-1}\alpha,\beta'=\beta P$. Thus $(A',B',\alpha',\beta',a,b)\in V$ and we may assume 
\[
A'=\begin{pmatrix}
    a'I_r& A_2\\
    0&A_4
\end{pmatrix},B'=\begin{pmatrix}
    B_1& B_2\\
    0&B_4
\end{pmatrix},\alpha'=\begin{pmatrix}
    \alpha_1\\
    0
\end{pmatrix},\beta'=(0,\beta_1),
\]
where $B_1\in M_{r}(\mathbb{K}),\alpha_1\in M_{r\times 1}(\mathbb{K}),\beta_1 \in M_{1\times (m-r)}(\mathbb{K})$.

If $\mathrm{rank}(A_2)<r$, then $(A',B',\alpha',\beta',a,b)\in \overline{\mathcal{CV}(m)^o}$ by Lemma \ref{closure case 3: rank(A_2)<r}.
If $A_4\neq a'I_{m-r}$, then $\ker(A_4-a'I_{m-r})\leq m-r-1<\frac{m}{2}<\ker(A'-a'I_{m})$, thus there is $\gamma=(\gamma',\gamma'')\in M_{1\times m}(\mathbb{K})$ such that $0\neq \gamma'\in M_{1\times r}(\mathbb{K})$ and $\gamma A'=a'\gamma$, thus $\gamma' A_2+\gamma''(A_4-a'I_{m-r})=0$.
Take $\alpha''\in M_{r\times 1}(\mathbb{K})$ such that $\gamma'\alpha''=1$. Let 
\[
Q=\begin{pmatrix}
    I_r& \alpha''\gamma''\\
    0&I_{m-r}
\end{pmatrix}.
\]
Thus 
\[
QA'Q^{-1}=\begin{pmatrix}
    a'I_r& A_2+ \alpha''\gamma''(A_4-a'I_{m-r})\\
    0&A_4
\end{pmatrix}, QB'Q^{-1}=\begin{pmatrix}
    B_1'&B_2'\\
    0&B_4'
\end{pmatrix}, 
\]
and $Q\alpha'=\alpha',\beta'Q^{-1}=\beta'$. Since $\gamma' A_2+\gamma''(A_4-a'I_{m-r})=0$ and $\gamma'\neq 0$, we have $\gamma'( A_2+ \alpha''\gamma''(A_4-a'I_{m-r}))=0$, thus $(QA'Q^{-1},QB'Q^{-1},\alpha',\beta',a,b)\in \overline{\mathcal{CV}(m)^o}$ by Lemma \ref{closure case 3: rank(A_2)<r}, hence $(A',B',\alpha',\beta',a,b)\in \overline{\mathcal{CV}(m)^o}$.

If $A_4= a'I_{m-r}$ and $\mathrm{rank}(A_2)=r$, then $\mathrm{rank}(A'-a'I_m)=r<\frac{m}{2}$, thus $m=2s+1$ and $r=s<m-r$ since $\frac{m}{2}-1<r$. Hence $(A',B',\alpha',\beta',a,b)\in\overline{\mathcal{CV}(m)^o}$ by Lemma \ref{closure case 4: A_4=a'}.
\end{proof}

Finally, we will show that $\mathrm{Spec}\Tilde{R}(m)/(t,w_{m})$ is regular at all generic points.

\begin{lemma}\label{Tilde(R)/(vu,t) regular point}
    The scheme  $\mathrm{Spec}\Tilde{R}(m)/(t,w_{m})$ is equidimensonal of dimension $m^2+m+2$ and is regular at all generic points.
\end{lemma}
\begin{proof}
The scheme $\mathrm{Spec}\Tilde{R}(m)/(t,w_m)$ is equidimensional of dimension $m^2+m+2$ by Lemma \ref{dim(im(psi))} and Lemma \ref{CV^o dense}.
For $0\leq m_1,m_2$ and $m_1+m_2\leq m$, 
let $P_{m_1m_2}=(A_{m_1},B_{m_2},\alpha_{m_1},\beta_{m_2},1,1)$ be a closed point of $\mathrm{Spec}\Tilde{R}(m)/(t,w_m)$, and
\[
A_{m_1}=\mathrm{diag}(a_1,\dots,a_{m}),B_{m_2}=\mathrm{diag}(b_1,\dots,b_{m}),
\]
\[
\alpha_{m_1}=(\gamma_{m_1},0)^t,\beta_{m_2}=(0,\gamma_{m_2}),\gamma_{i}=(1,\dots,1)\in \mathbb{K}^{i},
\]
and $1,a_{m_1+1},\dots,a_{m},b_{1},\dots,b_{m-m_2}$ be pairwise distinct and $a_i=b_j=1$ for $1\leq i\leq m_1$ and $m-m_2<j\leq m$. Thus each irreducible component of $\mathrm{Spec}R(m)/(t,w_m)$ contains $P_{m_1m_2}$ for some $m_1,m_2$ by Lemma \ref{CV^o dense}.
Obviously we have
\[
\Tilde{R}(m)/(t,w_m)=\mathbb{K}[x_{ij},y_{ij},u_i,v_i,t_1,t_2\mid 1\leq i,j\leq m]/J''
\]
where $J''=(f_{ij},g_i,h_i,w_m\mid 1\leq i,j\leq m)$ and  $f_{ij}=\sum_{r=1}^{m} x_{ir}y_{rj}-y_{ir}x_{rj}$, $(g_1,\dots,g_{m})^t=(X_{m}-t_1I_{m})(u_1,\dots,u_{m})^t$,  $(h_1,\dots,h_{m})=(v_1,\dots,v_{m})(Y_{m}-t_2I_{m})$.
Note that
 \[
\frac{\partial f_{ij}}{\partial x_{i'j'}}=\begin{cases}
    y_{jj}-y_{ii}& \text{if }i'=i,j'=j\\
 y_{j'j}& \text{if }i'=i,j'\neq j\\
 -y_{i'i}& \text{if }i'\neq i,j'=j\\
 0& \text{if }i'\neq i,j'\neq j\\
\end{cases},
\frac{\partial f_{ij}}{\partial y_{i'j'}}=\begin{cases}
 x_{ii}-x_{jj}& \text{if }i'=i,j'=j\\
 -x_{j'j}& \text{if }i'=i,j'\neq j\\
 x_{i'i}& \text{if }i'\neq i,j'=j\\
 0& \text{if }i'\neq i,j'\neq j\\
\end{cases},
\]
\[
\frac{\partial g_{i}}{\partial x_{i'j'}}=\begin{cases}
    u_{j'}& \text{if }i'=i\\
    0& \text{if }i'\neq i
\end{cases},
\frac{\partial g_{i}}{\partial u_{i'}}=\begin{cases}
    x_{ii'}-t_1& \text{if }i'=i\\
    x_{ii'}& \text{if }i'\neq i
\end{cases},
    \]
\[
\frac{\partial h_{i}}{\partial y_{i'j'}}=\begin{cases}
    v_{i'}& \text{if }j'=i\\
    0& \text{if }j'\neq i
\end{cases},
\frac{\partial h_{i}}{\partial v_{i'}}=\begin{cases}
    y_{i'i}-t_2& \text{if }i'=i\\
    y_{i'i}& \text{if }i'\neq i
\end{cases}.
    \]
Thus for any $0\leq m_1,m_2$ and $m_1+m_2\leq m$, we have
 \[
\frac{\partial f_{ij}}{\partial x_{i'j'}}(P_{m_1m_2})=\begin{cases}
    b_{j}-b_{i}\neq 0& \text{if }i'=i,j'=j,i\neq j, m-m_2\geq i\text{ or }j \\
 0& \text{otherwise }\\
\end{cases},\]
\[
\frac{\partial f_{ij}}{\partial y_{i'j'}}(P_{m_1m_2})=\begin{cases}
    a_{i}-a_{j}\neq 0& \text{if }i'=i,j'=j,i\neq j,m_1< i\text{ or }j \\
 0& \text{otherwise }\\
\end{cases},\]
\[
\frac{\partial g_{i}}{\partial x_{i'j'}}(P_{m_1m_2})=\begin{cases}
    1& \text{if }i'=i,1\leq j'\leq m_1\\
    0& \text{otherwise }
\end{cases},
\]
\[
\frac{\partial g_{i}}{\partial u_{i'}}(P_{m_1m_2})=\begin{cases}
    a_i-1\neq 0& \text{if }i'=i>m_1\\
    0& \text{otherwise }
\end{cases},
    \]
    \[
\frac{\partial h_{i}}{\partial y_{i'j'}}(P_{m_1m_2})=\begin{cases}
    1& \text{if }j'=i,m-m_2< i'\\
    0& \text{otherwise }
\end{cases},\]
\[
\frac{\partial h_{i}}{\partial v_{i'}}(P_{m_1m_2})=\begin{cases}
    b_i-1\neq 0& \text{if }i'=i\leq m-m_2\\
    0& \text{otherwise }
\end{cases}.
    \]
Let $J_{P_{m_1m_2}}\in M_{(m^2+2m+1)\times (2m^2+2m+2)}(\mathbb{K})$ be the Jacobian matrix. We have $\mathrm{rank}(J_{P_{m_1m_2}})\leq 2m^2+2m+2-(m^2+m+2)=m^2+m$.
For any $1\leq i\neq j,i'\neq j'\leq m$, let 
\[
\frac{\partial f_{ij}}{\partial z_{i'j'}}(P_{m_1m_2})=\begin{cases}
    \frac{\partial f_{ij}}{\partial x_{i'j'}}(P_{m_1m_2}) &\text{ if } m-m_2\geq i \text{ or }j\\
    \frac{\partial f_{ij}}{\partial y_{i'j'}}(P_{m_1m_2}) &\text{ if } i,j>m-m_2\\
\end{cases},
\]
\[
z_i=\begin{cases}
    x_{ii}& \text{ if }1\leq i\leq m_1\\
    u_i& \text{ if }i>m_1
\end{cases},\tau_i=\begin{cases}
    y_{ii}& \text{ if } i>m-m_2\\
    v_i& \text{ if }i\leq m-m_2
\end{cases}.
\]
Note that $J_{P_{m_1m_2}}$ has a submatrix $Q=(Q_{ij})_{1\leq i,j\leq 3}$,
where 
\begin{align*}
    & Q_{11}=\left(\frac{\partial f_{ij}}{\partial z_{i'j'}}(P_{m_1m_2})\right)_{1\leq i\neq j,i'\neq j'\leq m}\in GL_{m^2-m}(\mathbb{K}),\\
    & Q_{12}=\left(\frac{\partial f_{ij}}{\partial z_{i'}}(P_{m_1m_2})\right)_{1\leq i\neq j\leq m,1\leq i'\leq m}=0,\\
    & Q_{13}=\left(\frac{\partial f_{ij}}{\partial \tau_{i'}}(P_{m_1m_2})\right)_{1\leq i\neq j\leq m,1\leq i'\leq m}=0,\\
    &Q_{22}=\left(\frac{\partial g_{i}}{\partial z_{i'}}(P_{m_1m_2})\right)_{1\leq i,i'\leq m}\in GL_m(\mathbb{K}),\\
    &Q_{23}=\left(\frac{\partial g_{i}}{\partial \tau_{i'}}(P_{m_1m_2})\right)_{1\leq i,i'\leq m}=0,\\
    &Q_{33}=\left(\frac{\partial h_{i}}{\partial \tau_{i'}}(P_{m_1m_2})\right)_{1\leq i,i'\leq m}\in GL_m(\mathbb{K}).
\end{align*}
Thus $\mathrm{rank}(Q)=m^2+m$, hence $\mathrm{rank}(J_{P_{m_1m_2}})=m^2+m$. 
Hence $P_{m_1m_2}$ is a regular point of $\mathrm{Spec}\Tilde{R}(m)/(t,w_m)$.
 \end{proof}

This completes the proof of Lemma \ref{lemma: dim(Tilde(R)/(vu,t))}.

\subsection{Proof of Lemma \ref{lemma: dim(Tilde(R)/(t))}}\label{Proof of lemma: dim(Tilde(R)/(t))}\ 

We will divide the proof of Lemma \ref{lemma: dim(Tilde(R)/(t))} into the following several lemmas.

\begin{lemma}\label{Tilde(R)[t^(-1)]/(vu)) dense}
The scheme $\mathrm{Spec}\Tilde{R}(m)[t^{-1}]/(w_m)$ is dense in $\mathrm{Spec}\Tilde{R}(m)/(w_m)$. 
\end{lemma}
\begin{proof}
Let $\mathcal{MV}=\{(A,B,\alpha,\beta,a,b,c)\in M_{m}(\mathbb{K})^2\times M_{m\times 1}(\mathbb{K})\times M_{1\times m}(\mathbb{K})\times \mathbb{K}^3\mid AB-BA=c\alpha\beta,A\alpha=a\alpha,\beta B=b\beta,\beta\alpha=0\}$ and $\mathcal{MV}^o=\mathcal{MV}\setminus \mathcal{CV}(m)$.
For $0\leq m_1,m_2\leq m$ and $m_1+m_2\leq m$, 
and  
\[
A=\mathrm{diag}(a I_{m_1},a_1,\dots,a_{m-m_1}),B=\mathrm{diag}(b_{1},\dots,b_{m-m_2},
bI_{m_2}),
\]
and $\alpha=(\alpha',0)^t,\beta=(0,\beta')$, where $a,b,a_1,\dots,a_{m-m_1},b_1,\dots,b_{m-m_2}\in \mathbb{K}$ are pairwise distinct, $\alpha'=(\alpha_1',\dots,\alpha_{m_1}')\in \mathbb{K}^{m_1},,\beta'=(\beta_1',\dots,\beta_{m_2}')\in \mathbb{K}^{m_2}$, thus $(A,B,\alpha,\beta,a,b)\in \mathcal{CV}(m)$.
For any $c\in \mathbb{K}$, let 
\[
A(c)=A+\begin{pmatrix}
    0& \left(\frac{\alpha_i'\beta_j'}{b-b_i}c\right)_{1\leq i\leq m_1,1\leq j\leq m_2}\\
    0&0
\end{pmatrix},
\]
thus $A(c)B-BA(c)=c\alpha\beta$ and $A(c)\alpha=a\alpha$ since $m_1+m_2\leq m$, thus $(A(c),B,\alpha,\beta,a,b,c)\in \mathcal{MV}$.
Let $D=\mathbb{K}\setminus\{0\}$.
For any $c\in D$, we have $(A(c),B,\alpha,\beta,a,b,c)\in \mathcal{MV}^o$.
Since $D$ is dense in $\mathbb{K}$,  $(A,B,\alpha,\beta,a,b,0)\in\overline{\mathcal{MV}^o}$.
Thus $\mathcal{CV}(m)\subseteq\overline{\mathcal{MV}^o}$ by Lemma \ref{CV^o dense}. 
So $\mathrm{Spec}\Tilde{R}(m)[t^{-1}]/(w_m)$ is dense in $\mathrm{Spec}\Tilde{R}(m)/(w_m)$. 
\end{proof}

Let 
\[
\pi: \mathcal{U}=\mathrm{Spec}\Tilde{R}(m)/(w_m)\to \mathrm{Spec}\mathbb{K}[t]
\]
be the natural map, and $U=\mathrm{Spec}\mathbb{K}[t,t^{-1}]$ and $p_i$ be the closed point of $\mathrm{Spec}\mathbb{K}[t]$ corresponding to the ideal $(t-i)$ for $i=0,1$.
We can define $\mathbb{G}_m$-actions on $\mathcal{U}$ and $\mathrm{Spec}\mathbb{K}[t]$ in a natural way such that $U$ and ${\pi}^{-1}(U)$ are stable under these actions, and that $\pi$ is $\mathbb{G}_m$-equivariant. 
Similar to the proof of Lemma \ref{phi_i flat}, we can see that $\pi|_{\pi^{-1}(U)}$ is flat and $\pi^{-1}(p_1)\cong \pi^{-1}(p')$ for any closed point $p'\in U$.

\begin{lemma}\label{proof of dim(Tilde(R)/(vu))}
     The scheme $\mathrm{Spec}\Tilde{R}(m)/(w_m)$ is equidimensional of dimension $m^2+m+3$.
\end{lemma}
\begin{proof}
By Lemma \ref{lemma:connection_among_rings}, we have
\[
R(m+1)/(x_{i,m+1},y_{m+1,i}\mid 1\leq i\leq m)\cong \Tilde{R}(m)/(t-1).
\]
By Lemma \ref{lemma:Motzkin-Taussky}, the scheme $\mathrm{Spec}R(m+1)$ is irreducible of dimension $(m+1)^2+m+1$, hence each irreducible component of $\pi^{-1}(p_1)=\mathrm{Spec}\Tilde{R}(m)/(t-1)$ has dimension $\geq m^2+m+2$.

Let $Z$ be an irreducible component of $\mathcal{U}$. 
Since $\pi^{-1}(U)$ is dense in $\mathcal{U}$ by Lemma \ref{Tilde(R)[t^(-1)]/(vu)) dense}, we can take a closed point $q\in Z\cap \pi^{-1}(U)$ such that $\dim(\mathcal{O}_{Z,q})=\dim(\mathcal{O}_{\mathcal{U},q})$.
Since $\pi|_{\pi^{-1}(U)}$ is flat and $q\in \pi^{-1}(U)$, we have
\[
\dim(Z)=\dim(\mathcal{O}_{\mathcal{U},q})=\dim(\mathcal{O}_{\pi^{-1}(\pi(q)),q})+1\geq m^2+m+3.
\]

On the other hand, the ring $\Tilde{R}(m)/(w_m)=\bigoplus_{i=0}^{\infty}\Tilde{R}_i$ is graded, where $\Tilde{R}_i$ is the set of the image of 
\[
\sum_{i_1+\cdots+i_j=i,i_k\geq 0}f_{i_1\dots i_j}\prod_{k=1}^jz_{k}^{i_k}
\]
in $\Tilde{R}(m)/(w_m)$ where $f_{i_1\dots i_j}\in \mathbb{K}[u_1,\dots,u_m]$ and 
\[
z_k\in\{x_{11},\dots,x_{mm},y_{11},\dots,y_{mm},v_1,\dots,v_m,t_1,t_2,t\}.
\]
By Lemma \ref{lemma: prime divisor homogeneous}, every minimal prime ideal of $\Tilde{R}(m)/(w_m)$ is homogeneous, thus it is contained in the irrelevant ideal $\bigoplus_{i>0}\Tilde{R}_i$.
Hence $Z\cap \pi^{-1}(p_{0})\neq \emptyset$.
By Lemma \ref{Tilde(R)/(vu,t) regular point}, we have $\dim\pi^{-1}(p_0)=m^2+m+2$, thus
\[
\dim(Z)\leq \dim\pi^{-1}(p_0)+1=m^2+m+3.
\]
Hence $\dim(Z)=m^2+m+3$.
So $\mathcal{U}$ is equidimensional of dimension $m^2+m+3$.
\end{proof}

\begin{lemma}\label{proof of dim(Tilde(R)/(t-1))}
     The scheme  $\mathrm{Spec}\Tilde{R}(m)/(t-1)$ is equidimensional of dimension $m^2+m+2$ and is regular at all generic points.
\end{lemma}
\begin{proof}
By Lemma \ref{proof of dim(Tilde(R)/(vu))}, the scheme $\mathcal{U}$ is equidimensional of dimension $m^2+m+3$. Since $\pi|_{\pi^{-1}(U)}$ is flat, the fiber $\pi^{-1}(p_1)=\mathrm{Spec}\Tilde{R}(m)/(t-1)$ is equidimensional of dimension $m^2+m+2$.

Let $Z$ be an irreducible component of $\mathcal{U}$. 
By Lemma \ref{Tilde(R)/(vu,t) regular point} and the proof of Lemma \ref{proof of dim(Tilde(R)/(vu))}, we have $\dim(Z\cap \pi^{-1}(p_0))\geq m^2+m+2=\dim\pi^{-1}(p_0)$.
Thus $Z$ contains an irreducible component of $\pi^{-1}(p_0)$.
Let $\varphi:\pi^{-1}(p_0)\to \mathcal{U}$ be the natural map, and let $P'_{m_1m_2}=\varphi(P_{m_1m_2})$ for $m_1,m_2\geq 0$ and $m_1+m_2\leq m$, where $P_{m_1m_2}$ is defined in the proof of Lemma \ref{Tilde(R)/(vu,t) regular point}. 
Similar to the proof of Lemma \ref{Tilde(R)/(vu,t) regular point}, we see that 
$\mathcal{U}$ is regular at $P'_{m_1m_2}$, thus is regular at all generic points.

By Lemma \ref{Tilde(R)[t^(-1)]/(vu)) dense}, $\pi^{-1}(U)$ is dense in $\mathcal{U}$, therefore $Z'$ is an irreducible component of $\mathcal{U}$ if and only if $Z'\cap \pi^{-1}(\eta)$ is an irreducible component of $\pi^{-1}(\eta)$, where $\eta$ is the generic point of $\mathrm{Spec}\mathbb{K}[t]$.
Let $\eta'$ be the generic point of $Z$. Thus $\eta'\in \pi^{-1}(\eta)$ and $\mathcal{O}_{\mathcal{U},\eta'}$ is regular, and therefore $\mathcal{O}_{\pi^{-1}(\eta),\eta'}=S^{-1}\mathcal{O}_{\mathcal{U},\eta'}$ is regular, where $S=\mathbb{K}[t]\setminus \{0\}$.
So the generic points of $\pi^{-1}(\eta)$ are regular.
Since $\pi^{-1}(\eta)\cong \pi^{-1}(p_1)\times_{\mathbb{K}}\mathrm{Spec}\mathbb{K}(t)$, the generic points of $\pi^{-1}(p_1)$ are regular.
\end{proof}

\begin{lemma}\label{proof of dim(Tilde(R)/t))}
     The scheme  $\mathrm{Spec}\Tilde{R}(m)/(t)$ is equidimensional of dimension $m^2+m+2$ and is regular at all generic points.
\end{lemma}
\begin{proof}

First, we will show that $\mathrm{Spec}\Tilde{R}(m)[w_m^{-1}]/(t)$ is equidimensional of dimension $m^2+m+2$ and is regular at all generic points.

Let $U'=\mathrm{Spec}\mathbb{K}[u_1,\dots,u_m,v_1,\dots,v_m,(\sum_{i=1}^mu_iv_i)^{-1}]$ and
\[
\pi': \mathcal{U}'=\mathrm{Spec}\Tilde{R}(m)[w_m^{-1}]/(t)\to U'
\]
be the natural map.
Let $ GL_{m}\times \mathbb{G}_m$ act on $\mathcal{U}'$ according to the rule $(g,c)\cdot (A,B,\alpha,\beta,a,b)= (cgAg^{-1},gBg^{-1},cg\alpha,\beta g^{-1},ca,b)$
for $(A,B,\alpha,\beta,a,b)\in\mathcal{U}'(\mathbb{K})$ and $ GL_{m}\times \mathbb{G}_m$ act on $U'$ according to the rule 
$(g,c)\cdot (\alpha,\beta)= (cg\alpha,\beta g^{-1})$
for $(\alpha,\beta)\in U'(\mathbb{K})$. 
Similar to the proof of Lemma \ref{phi_i flat}, we can see that $\pi'$ is flat and $\pi'^{-1}(p)\cong \pi'^{-1}(p')$ for any closed point $p'\in U'$, where $p$ is the closed point of $U'$ corresponding to the ideal $(u_{i},v_i,u_m-1,v_m-1\mid 1\leq i\leq m-1)$.
Hence $\pi'^{-1}(p)\cong \mathrm{Spec}R(m)/(x_{im},y_{mi}\mid 1\leq i\leq m-1)$. By Lemma \ref{lemma:connection_among_rings}, the fiber
\[
\pi'^{-1}(p)\cong \mathrm{Spec}\Tilde{R}(m-1)/(t-1),
\]
thus 
$\pi'^{-1}(p)$ is equidimensional of dimension $m^2-m+2$ and is regular at all generic points by Lemma \ref{proof of dim(Tilde(R)/(t-1))}.

Let $Z$ be an irreducible component of $\mathcal{U}'$.
We can take a closed point $q\in Z$ such that $\dim(\mathcal{O}_{Z,q})=\dim(\mathcal{O}_{\mathcal{U}',q})$ and $\pi'^{-1}(p)\cong \pi'^{-1}(\pi'(q))$.
Since $\pi'$ is flat, we have
\[\dim(Z)=\dim(\mathcal{O}_{\mathcal{U}',q})=\dim(\mathcal{O}_{\pi'^{-1}(\pi'(q)),q})+2m= m^2+m+2.
\]
Hence $\mathcal{U}'$ is equidimensional of dimension $m^2+m+2$.

Since $\dim(Z\cap \pi'^{-1}(\pi'(q)))\geq m^2-m+2$ and $\pi'^{-1}(p)\cong \pi'^{-1}(\pi'(q))$, $Z$ contains an irreducible component $Z'$ of $\pi'^{-1}(\pi'(q))$.
Let $p'\in Z'\subseteq Z$ be a regular point of $\pi'^{-1}(\pi'(q))$. Since $\mathcal{O}_{\pi'^{-1}(\pi'(q)),p'}$ and $\mathcal{O}_{U',\pi'(q)}$ are regular and $\pi'$ is flat, we see that $\mathcal{O}_{\mathcal{U}',p'}$ is regular by \cite[Theorem 23.7]{Matsumura}.
So $\mathcal{U}'$ is regular at all generic points.

Finally, since $\mathcal{U}'$ is equidimensional of dimension $m^2+m+2$ and
    \[
    \mathrm{Spec}\Tilde{R}(m)/(t)=\mathrm{Spec}\Tilde{R}(m)/(t,w_m)\bigcup \mathcal{U}',
    \]
the scheme $\mathrm{Spec}\Tilde{R}(m)/(t)$ is equidimensional of dimension $m^2+m+2$ by Lemma \ref{Tilde(R)/(vu,t) regular point}. Let $\eta$ be a generic point of $\mathrm{Spec}\Tilde{R}(m)/(t)$.
If $\eta\in \mathcal{U}'$, then since $\eta$ is a regular point of $\mathcal{U}'$, so is $\mathrm{Spec}\Tilde{R}(m)/(t)$.
If $\eta\in\mathrm{Spec}\Tilde{R}(m)/(t,w_m)$, we can take a closed point $P_{m_1m_2}$ of $\mathrm{Spec}\Tilde{R}(m)/(t)$, which is defined in the proof of Lemma \ref{Tilde(R)/(vu,t) regular point} and is in $\overline{\{\eta\}}$. Similar to
the proof of Lemma \ref{Tilde(R)/(vu,t) regular point}, $P_{m_1m_2}$ is a regular point of $\mathrm{Spec}\Tilde{R}(m)/(t)$. 
Hence $\mathrm{Spec}\Tilde{R}(m)/(t)$ is regular at all generic points.
\end{proof}

This completes the proof of Lemma \ref{lemma: dim(Tilde(R)/(t))}.
 
\begin{remark}(Summary)
We have proved that the commuting scheme $\mathfrak{C}^2_{\mathfrak{gl}_n}$ is Cohen–Macaulay and normal by induction on $n$ and establishing connections with several intermediate rings. In doing so, we  considered the $\mathfrak{gl}_n$ case. It would be natural to extend this construction directly to $\mathfrak{C}^2_{\mathfrak{so}_n}$, however there would be obstructions. For instance, let $\mathrm{char}\  \mathbb{K}=0$, $n>2$ and $\mathcal{O}_{\mathfrak{C}^2_{\mathfrak{so}_n}}(\mathfrak{C}^2_{\mathfrak{so}_n})=\mathbb{K}[x_{ij},y_{ij}]/I$, where $I$ is the ideal generated by all of the entries of the matrices $[X_n,Y_n]$, $X_n+X_n^t$ and $Y_n+Y_n^t$, where $X_n=(x_{ij})_{1\leq i,j\leq n},Y_n=(y_{ij})_{1\leq i,j\leq n}$. 
Consider the quotient ring $H=\mathcal{O}_{\mathfrak{C}^2_{\mathfrak{so}_n}}(\mathfrak{C}^2_{\mathfrak{so}_n})/(x_{in}\mid 1\leq i\leq n-1)$. Then the set of closed points of $\mathrm{Spec}H$ is isomorphic to $\mathcal{H}=\{(A,B,\alpha)\in \mathfrak{so}_{n-1}^2\times M_{(n-1)\times 1}(\mathbb{K})\mid AB=BA,A\alpha=0\}$.
We can prove that $\mathrm{Spec}H$ is not equidimensional. This is since the closed subset $\{(A,B,\alpha)\in \mathcal{H}\mid A=0\}$ has dimension $\frac{(n-1)(n-2)}{2}+n-1$, and the open subset $\{(A,B,\alpha)\in \mathcal{H}\mid A \text{ has only simple characteristic roots}\}$ has dimension $\frac{(n-1)(n-2)}{2}+\lceil \frac{n-1}{2}\rceil$, where $\lceil r\rceil=\min\{a\in \mathbb{Z}\mid a\geq r\}$ for any $r\in \mathbb{R}$.
There might be other quotient rings which can serve as intermediate rings to complete the induction step, however we have not yet found such suitable quotient rings. We would postpone this investigation to a sequel to this article. Finally, considering the image of the Hitchin morphism in positive characteristic might be an interesting venue to explore in the future. 
\end{remark}

\end{document}